\documentclass{article}
\usepackage{amsfonts}
\usepackage{latexsym,amsmath,amsthm,amssymb,cite}
\usepackage[sans]{dsfont}
\usepackage{mathpazo}
\usepackage{bbm}
\usepackage{color}
\usepackage[bookmarksnumbered, colorlinks, plainpages]{hyperref}

\hypersetup{citecolor={blue}}
\numberwithin{equation}{section}
\newtheorem{theorem}{Theorem}[section]
\newtheorem{lemma}[theorem]{Lemma}
\newtheorem{corollary}[theorem]{Corollary}

\theoremstyle{definition}
\newtheorem{definition}[theorem]{Definition}

\newtheorem{remark}[theorem]{Remark}

\begin{document}

\begin{center}
{\Large \textbf{Langevin equation with nonlocal boundary conditions involving a
$\psi$--Caputo fractional  operator}%
}

\vskip.30in

{\large \textbf{Arjumand Seemab$^{1}$, Jehad Alzabut$^{2,\ast }$, Mujeeb ur
Rehman$^{1}$\\ Yacine Adjabi$^{3}$, Mohammed S. Abdo$^{4}$}} \\[2mm]
{\small $^{1}$School of Natural Sciences, National University of Sciences
and Technology, Islamabad Pakistan.\\ Email: mrehman@sns.nust.edu.pk \\[%
0pt]
$^{2 }$Department of Mathematics and General Sciences, Prince Sultan
University, 11586 Riyadh, Saudi Arabia.\\ Email: jalzabut@psu.edu.sa.\\[0pt]
$^{3}$Department of Mathematics, Faculty of Sciences University of M'hamed
Bougara, UMBB, Algeria.\\ Email: adjabiy@yahoo.fr\\[0pt]
$^{4}$Department of Mathematics, Hodeidah University, Al-Hodeidah, Yemen. \\
 Email: msabdo1977@gmail.com\\[0pt]
}
\end{center}

\vskip 4mm

\noindent {\small \textbf{Abstract.}~ This paper studies Langevin equation with nonlocal boundary conditions involving a
$\psi$--Caputo fractional derivatives operator. By the aide of fixed point techniques of Krasnoselskii and Banach, we derive new results on existence and uniqueness of the problem at hand. Further, the $\psi $-fractional Gronwall inequality and $\psi $--fractional integration by parts are employed to prove Ulam--Hyers  and Ulam--Hyers--Rassias  stability for the solutions. Examples are gifted to demonstrate the advantage of our major results. The proposed results here are more general than the existing results in the literature and obtain them as particular cases.}

\vskip.15in \footnotetext{\textbf{AMS 2010 Mathematics Subject Classification}
: $34$A$08$,\;$26$A$33$,\;  $34$A$12$,\; $34$D$20$.} \footnotetext{\textbf{\ Keywords}:
Generalized fractional operators; $\psi $--Caputo derivative; $\psi $%
--fractional Langevin type equation;  Existence and uniqueness; U-H stability type; Krasnoselskii fixed
point theorem; $\psi $--fractional Gronwall inequality.}

\footnotetext{\textbf{$^{*}$ Corresponding author: email}:
jalzabut@psu.edu.sa.} \baselineskip=14pt

\setcounter{section}{0} \numberwithin{equation}{section}

\section{Introduction}

\qquad Lately, fractional calculus has played a very significant role in various scientific fields; see for instance  \cite{KilbasST,Herrmann} and the references cited therein. As a result of this, fractional differential equations have caught the attention of many investigators working in different desciplines \cite{Podlubny,Diethelm}. However, most of researchers works have been conducted by using  fractional derivatives that mainly rely on Riemann--Liouville, Hadamard, Katugampola, Grunwald Letnikov and Caputo approaches.

Fractional derivatives of a function with respect to another function have
been considered in the classical monographs  \cite{SamkoKM,KilbasST} as a generalization of Riemann--Liouville derivative.
This fractional derivative is different from the other classical fractional
derivative as the kernel appears in terms of another function $\psi $. We will call this derivative as $\psi $--fractional derivative.\ Recently, this derivative has been
reconsidered by Almeida in \cite{Almeida} where the Caputo--type
regularization of the existing definition and some interesting properties
are provided. Several properties of this operator could be found in \cite%
{KilbasST,SamkoKM,Osler,Kiryakova,Agrawal}. For some particular cases of $\psi $%
, one can realize that $\psi $--fractional derivative can be reduced to the Caputo fractional derivative \cite{KilbasST}, the
Caputo--Hadamard fractional derivative \cite{JaradBT} and the Caputo--Erd%
\'{e}lyi--Kober fractional derivative \cite{LuchkoT}.

On the other hand, the investigation of  qualitative properties of solutions for different fractional differential (and integral ) equations is the key theme of applied mathematics research. Numerous interesting results concerning the existence, uniqueness, multiplicity, and stability of solutions or positive solutions by applying some fixed point techniques. However, most of the proposed problems have been handled concerning the classical fractional derivatives of the Riemann--Liouville and Caputo \cite%
{AhmadME,DarwichN,ObukhovskiiZA,ElShahed,Elsayed,YanST,ZhaoSHL,QinZL,SakthivelRM,z1,z2}.

In parallel with the intense investigation of fractional derivative, a normal generalization of the Langevin differential equation appears to be replacing the classical derivative by a fractional derivative to produce fractional Langevin equation (FLE). FLE was first introduced in \cite{MainradiP} and then
different types of FLE were the object of many scholars  \cite%
{AhmadAS,KiataramkulSTK,SudsutadNT, LiSS,YukunthornNT,zhouQ,ZhouAY}.
 In \cite{AhmadNAE}, the authors studied a nonlinear FLE involving two fractional orders at various intervals with three-point boundary conditions. FLE involving a Hadamard derivative type was considered in \cite{KiataramkulSTK,SudsutadNT}.

Alternatively, the stability problem of differential equations was discussed by Ulam in \cite{Ulam}). Thereafter, Hyers in \cite{Hyers} developed the concept of Ulam stability in the case of Banach spaces. Rassias provided a fabulous generalization of the Ulam--Hyers (U--H) stability of mappings by taking into account variables. His approach was refered to as Ulam--Hyers--Rassias (U--H--R) stability \cite{Rassias78}. Recently, the Ulam stability problem of implicit differential equations was extended into fractional implicit differential equations by some authors \cite {WangL,Jung,zada1,zada2}.
A series of papers was devoted to the investigation of existence, uniqueness
and U-H stability of solutions of the FLE within different kinds of
fractional derivatives.

Motivated by the recent developments in $\psi $%
--fractional calculus, in the present work, we investigate the existence,
uniqueness and stability in the sense U-H-R of solutions for the
following FLE within $\psi $--Caputo fractional derivatives involving nonlocal boundary
 conditions%
\begin{equation}
\left\{
\begin{array}{l}
\left( ^{c}D_{a+,t}^{\varrho ,\psi }\right) \left( ^{c}D_{a+,t}^{\varsigma ,\psi
}+\lambda \right) \left[ u\right] =f(t,u(t),^{c}D_{a+,t}^{\delta ,\psi }%
\left[ u\right] ),~t\in \left( a,T\right) ,\medskip \\
u(a)=0,~u(\eta )=0,~u(T)=\mu \left( J_{a+,\xi }^{\delta ,\psi }\right) \left[
u\right] ,~\lambda ,\mu >0,%
\end{array}%
\right.  \label{EQ1}
\end{equation}%
where $\left( J_{a+,\xi }^{\delta ,\psi }\right) $ and $\left(
^{c}D_{a+,t}^{\theta ,\psi }\right) $\ are $\psi $-fractional integral of
order $\delta $,$\ \psi $-Caputo fractional derivative of order $\theta \in
\{\varrho ,\varsigma ,\delta \}$ respectively, $0\leq a<\eta <\xi <T<\infty ,\
1<\varrho \leq 2,\ 0<\delta <\varsigma \leq 1\ $and $f:[a,T]\times \mathbb{R}%
\times \mathbb{R}\rightarrow \mathbb{\ R}^{+}$ is a continuous function. It is worth mentioning here that
the proposed results in this paper which rely on
$\psi $-fractional integrals and $\psi $-Caputo fractional derivatives
can generalize the existing results in the literature  and obtain them as particular cases.

The major contributions of the work are as follows: Some lemmas and definitions on $\psi $-fractional calculus theory are recalled in Section $2$. In Section $3$, we prove the existence and uniqueness of solutions for problem (\ref{EQ1}) via applying fixed point theorems. Section $4$ devotes to discuss different types of stability results for the problem (\ref{EQ1}) by the aid of $\psi $-Gronwall's inequality and $\psi $--fractional integration by parts.
Examples are given in Section $5$ to check the applicability of the theoretical findings.\
We end the paper by a conclusion in Section $6.$

\section{Preliminaries and essential Lemmas}

The standard Riemann--Liouville fractional integral of order $Re(\varrho )>0$%
, namely%
\begin{equation}
\left( J_{a+,t}^{\varrho }\right) \left[ u\right] =\frac{1}{\Gamma \left(
\varrho \right) }\int_{a}^{t}\left( t-\tau \right) ^{\varrho -1}u\left( \tau
\right) \mathrm{d}\tau ,\ t>a.  \label{EQ2}
\end{equation}

The left-sided factional integrals and fractional derivatives of a fuction $%
u $ with respect to another function $\psi $ in the sense of
Riemann-Liouville are defined as follows \cite{Osler,Almeida}%
\begin{equation}
\left( J_{a+,t}^{\varrho ,\psi }\right) \left[ u\right] =\frac{1}{\Gamma
\left( \varrho \right) }\int_{a}^{t}\psi ^{\prime }\left( \tau \right) \left(
\psi \left( t\right) -\psi \left( \tau \right) \right) ^{\varrho -1}u\left(
\tau \right) \mathrm{d}\tau  \label{EQ3}
\end{equation}%
and%
\begin{equation}
\left( D_{a+,t}^{\varrho ,\psi }\right) \left[ u\right] =\left( \frac{1}{\psi ^{\prime }\left( t\right) }%
\frac{\mathrm{d}}{\mathrm{d}t}\right) ^{n}\left( J_{a+,t}^{n-\varrho ,\psi }\right) \left[ u\right] ,  \label{EQ4}
\end{equation}%
respectively, where $n=[\varrho ]+1$

Analogous formulas can be offered for the right fractional (integral and derivative) as follows:
\begin{equation}
\left( J_{t,b-}^{\varrho ,\psi }\right) \left[ u\right] =\frac{1}{\Gamma
\left( \varrho \right) }\int_{t}^{b}\psi ^{\prime }\left( \tau \right) \left(
\psi \left( \tau \right) -\psi \left( t\right) \right) ^{\varrho -1}u\left(
\tau \right) \mathrm{d}\tau ,  \label{EQ5}
\end{equation}%
and%
\begin{equation}
\left( D_{t,b-}^{\varrho ,\psi }\right) \left[ u\right] =\left( -\frac{1}{
\psi ^{\prime }\left( t\right) }\frac{\mathrm{d}}{\mathrm{d}t}\right)
^{n}\left( J_{t,b-}^{n-\varrho ,\psi }\right) \left[ u\right] .  \label{EQ6}
\end{equation}

The left (right) $\psi $-Caputo fractional derivatives of $u$ of order $\varrho $ are
given by%
\begin{equation}
\left( ^{c}D_{a+,t}^{\varrho ,\psi }\right) \left[ u\right] =\left( J_{a+,t}^{n-\varrho ,\psi }\right) \left( \frac{1}{\psi ^{\prime }\left( t\right) }%
\frac{\mathrm{d}}{\mathrm{d}t}\right) ^{n} \left[ u\right]  \label{EQ7}
\end{equation}%
and
\begin{equation}
\left( ^{c}D_{t,b-}^{\varrho ,\psi }\right) \left[ u\right] =\left( J_{t,b-}^{n-\varrho ,\psi } \right) \left( \frac{-1}{\psi ^{\prime }\left( t\right) }%
\frac{\mathrm{d}}{\mathrm{d}t}\right) ^{n} \left[ u\right],
\label{EQ8}
\end{equation}%
respectively.
In particular, when $\varrho \in (0,1)$, we have%
\begin{equation}
\left( ^{c}D_{a+,t}^{\varrho ,\psi }\right) \left[ u\right] =\left( J_{a+,t}^{1-\varrho ,\psi }\right) \left( \frac{1}{\psi ^{\prime }\left( t\right) }%
\frac{\mathrm{d}}{\mathrm{d}t}\right) \left[ u\right]
\label{EQ9}
\end{equation}%
and%
\begin{equation}
\left( ^{c}D_{t,b-}^{\varrho ,\psi }\right) \left[ u\right] =-\left( J_{t,b-}^{1-\varrho ,\psi }\right) \left( \frac{1}{\psi ^{\prime }\left( t\right) }%
\frac{\mathrm{d}}{\mathrm{d}t}\right) \left[ u\right] ,
\label{EQ10}
\end{equation}%
where $u,\psi \in \mathcal{C}^{n}[a,b]$ two functions such that $\psi $ is
increasing and $\psi ^{\prime }(t)\neq 0$, for all $t\in \lbrack a,b]$.

We propose the remarkable paper \cite{Agrawal} in which some generalizations using $\psi$-fractional
integrals and derivatives are described.
In particular, we have

\begin{equation*}
\left\{
\begin{array}{c}
if \ \ \psi (t)\longrightarrow t,\text{ \ \ then \ }J_{a+,t}^{\varrho ,\psi }\longrightarrow J_{a+,t}^{\varrho
},\qquad \qquad \qquad \qquad \ \ \\
if \ \ \psi (t)\longrightarrow \ln t,\text{ \ \ then \ }J_{a+,t}^{\varrho ,\psi }\longrightarrow \text{ }%
^{H}J_{a+,t}^{\varrho }, \qquad \qquad \qquad \ \ \\
if \ \ \psi (t)\longrightarrow t^{\rho }, \text{ \ \ then },\text{ \ }J_{a+,t}^{\varrho ,\psi }\longrightarrow
\text{ }^{\rho }J_{a+,t}^{\varrho }, \ \rho >0, \qquad \qquad \
\end{array}%
\right.
\end{equation*}%
where $J_{a+,t}^{\varrho },^{H}J_{a+,t}^{\varrho },^{\rho }J_{a+,t}^{\varrho }$
are classical Riemann--Liouville, Hadamard, and Katugampola fractional
operators.

\begin{lemma}
\label{L1}\cite{Almeida}Given $u\in \mathcal{C}([a,b])$ and $v\in \mathcal{C}
^{n}([a,b])$, we have that for all $\varrho >0$
\begin{eqnarray*}
\int_{a}^{b}v\left( \tau \right) \left( ^{c}D_{a+,\tau }^{\varrho ,\psi
}\right) \left[ u\right] \mathrm{d}\tau &=&\int_{a}^{b}u\left( \tau \right)
\left( ^{c}D_{\tau ,b-}^{\varrho ,\psi }\right) \left[ \frac{v}{\psi ^{\prime
}}\right] \frac{\mathrm{d}}{\mathrm{d}\tau }\psi \left( \tau \right) \mathrm{%
\ d}\tau \\
&&+\left. \sum_{k=0}^{n-1}\left( -\frac{1}{\psi ^{\prime }\left( t\right) }
\frac{\mathrm{d}}{\mathrm{d}t}\right) ^{k}\left( J_{\tau ,b-}^{n-\varrho
,\psi }\right) \left[ \frac{v}{\psi ^{\prime }}\right] u_{\psi }^{\left[
n-k-1\right] }\left( \tau \right) \right \vert _{\tau =a}^{\tau =b},
\end{eqnarray*}
where
\begin{equation*}
u_{\psi }^{\left[ k\right] }\left( t\right) =\left( \frac{1}{\psi ^{\prime
}\left( t\right) }\frac{\mathrm{d}}{\mathrm{d}t}\right) ^{k}u\left( t\right)
.
\end{equation*}
\end{lemma}

\begin{lemma}
\label{L2}Let $\varrho >0$ and $u,\ \psi \in \mathcal{C}([a,b]$. Then$\ $
\begin{equation*}
\left \Vert J_{a+,t}^{\varrho ,\psi }\left[ u\right] \right \Vert _{C}\leq
K_{\psi }\left \Vert u\right \Vert _{C},\ K_{\psi }=\frac{1}{\Gamma \left(
1+\varrho \right) }\left( \psi \left( b\right) -\psi \left( a\right) \right)
^{\varrho }.
\end{equation*}

For all $n-1<\varrho <n$,
\begin{equation*}
\left \Vert ^{c}D_{a+,t}^{\varrho ,\psi }\left[ u\right] \right \Vert
_{C}\leq K_{\psi }\left \Vert u\right \Vert _{C^{\left[ n\right] }\psi },\
K_{\psi }=\frac{1}{\Gamma \left( n+1-\varrho \right) }\left( \psi \left(
b\right) -\psi \left( a\right) \right) ^{n-\varrho }.
\end{equation*}
\end{lemma}

\begin{lemma}
\label{L3}Let $\varrho ,\varsigma >0$ , consider the functions \cite{Almeida}
\begin{equation*}
\left( J_{a+,t}^{\varrho ,\psi }\right) \left[ \left( \psi \left( \tau
\right) -\psi \left( a\right) \right) ^{\varsigma -1}\right] =\frac{\Gamma
\left( \varsigma \right) }{\Gamma \left( \varrho +\varsigma \right) }\left( \psi
\left( t\right) -\psi \left( a\right) \right) ^{\varrho +\varsigma -1},
\end{equation*}%
\begin{equation*}
\left( J_{a+,t}^{\varrho ,\psi }\right) \left[ 1\right] =\frac{1}{\Gamma
\left( 1+\varrho \right) }\left( \psi \left( t\right) -\psi \left( a\right)
\right) ^{\varrho }
\end{equation*}%
and
\begin{equation*}
\left( ^{c}D_{a+,t}^{\varrho ,\psi }\right) \left[ \left( \psi \left( \tau
\right) -\psi \left( a\right) \right) ^{\varsigma -1}\right] =\frac{\Gamma
\left( \varsigma \right) }{\Gamma \left( \varsigma -\varrho \right) }\left( \psi
\left( t\right) -\psi \left( a\right) \right) ^{\varsigma -\varrho -1},
\end{equation*}%
\begin{equation*}
\left( ^{c}D_{a+,t}^{\varrho ,\psi }\right) \left[ 1\right] =\frac{1}{\Gamma
\left( 1-\varrho \right) }\left( \psi \left( t\right) -\psi \left( a\right)
\right) ^{-\varrho },\text{ }\varrho >0.
\end{equation*}%
Note that
\begin{equation*}
\left( ^{c}D_{a+,t}^{\varrho ,\psi }\right) \left[ \left( \psi \left( \tau
\right) -\psi \left( a\right) \right) ^{k}\right] =0,\ k=0,..,n-1.
\end{equation*}
\end{lemma}

The subsequent properties are valid: if $%
\varrho ,\varsigma >0$, then%
\begin{equation}
\left( J_{a+,t}^{\varrho ,\psi }\right) \left( J_{a+,t}^{\varsigma ,\psi }\right)
\left[ u\right] =\left( J_{a+,t}^{\varrho +\varsigma ,\psi }\right) \left[ u%
\right] \text{ and\ }\left( ^{c}D_{a+,t}^{\varrho ,\psi }\right) \left(
^{c}D_{a+,t}^{\varsigma ,\psi }\right) \left[ u\right] =\left(
^{c}D_{a+,t}^{\varrho +\varsigma ,\psi }\right) \left[ u\right] ,  \label{EQ11}
\end{equation}%
\begin{equation}
\left( ^{c}D_{a+,t}^{\varrho ,\psi }\right) \left( J_{a+,t}^{\varsigma ,\psi
}\right) \left[ u\right] =\left( J_{a+,t}^{\varsigma -\varrho ,\psi }\right) %
\left[ u\right] \text{.}  \label{EQ12}
\end{equation}

\begin{lemma}
\label{L4}Given a function $u\in \mathcal{C}^{n}[a,b]$ and $\varrho >0$, we
have
\begin{equation*}
J_{a+,t}^{\varrho ,\psi }\left( ^{c}D_{a+,t}^{\varrho ,\psi }\right) \left[ u%
\right] =u\left( t\right) -\sum_{j=0}^{n-1}\left[ \frac{1}{j!}\left( \frac{1%
}{\psi ^{\prime }\left( t\right) }\frac{\mathrm{d}}{\mathrm{d}t}\right)
^{j}u\left( a\right) \right] \left( \psi \left( t\right) -\psi \left(
a\right) \right) ^{j}.
\end{equation*}%
In particular, given $\varrho \in (0,1)$, we have
\begin{equation*}
J_{a+,t}^{\varrho ,\psi }\left( ^{c}D_{a+,t}^{\varrho ,\psi }\right) \left[ u%
\right] =u\left( t\right) -u\left( a\right) .
\end{equation*}
\end{lemma}

\begin{lemma}
\label{L5}Given a function $u\in \mathcal{C}^{n}[a,b]$ and $1>\varrho >0$, we
have
\begin{equation*}
\left\Vert \left( J_{a+,t_{2}}^{\varrho ,\psi }\right) \left[ u\right]
-\left( J_{a+,t_{1}}^{\varrho ,\psi }\right) \left[ u\right] \right\Vert \leq
\frac{2\left\Vert u\right\Vert _{\infty }}{\Gamma \left( \varrho +1\right) }%
\left( \psi \left( t_{2}\right) -\psi \left( t_{1}\right) \right) ^{\varrho }.
\end{equation*}
\end{lemma}

\begin{proof}
	\begin{eqnarray*}
		\left \vert \left( J_{a+,t_{2}}^{\varrho ,\psi }\right) \left[ u\right]
		-\left( J_{a+,t_{1}}^{\varrho ,\psi }\right) \left[ u\right] \right \vert &=&%
		\frac{1}{\Gamma \left( \varrho \right) }\left \vert \int_{a}^{t_{2}}\psi
		^{\prime }\left( \tau \right) \left[ \left( \psi \left( t_{2}\right) -\psi
		\left( \tau \right) \right) ^{\varrho -1}-\left( \psi \left( t_{1}\right)
		-\psi \left( \tau \right) \right) ^{\varrho -1}\right] u\left( \tau \right)
		\mathrm{d}\tau \right \vert \\
		&&+\frac{1}{\Gamma \left( \varrho \right) }\left \vert
		\int_{t_{1}}^{t_{2}}\psi ^{\prime }\left( \tau \right) \left( \psi \left(
		t_{2}\right) -\psi \left( \tau \right) \right) ^{\varrho -1}u\left( \tau
		\right) \mathrm{d}\tau \right \vert \\
		&\leq &\frac{\left \Vert u\right \Vert _{\infty }}{\Gamma \left( \varrho
			+1\right) }\left[ \left( \psi \left( t_{2}\right) -\psi \left( t_{1}\right)
		\right) ^{\varrho }+\left( \psi \left( t_{1}\right) -\psi \left( a\right)
		\right) ^{\varrho }-\left( \psi \left( t_{2}\right) -\psi \left( a\right)
		\right) ^{\varrho }\right] \\
		&&+\frac{\left \Vert u\right \Vert _{\infty }}{\Gamma \left( \varrho
			+1\right) }\left( \psi \left( t_{2}\right) -\psi \left( t_{1}\right) \right)
		^{\varrho } \\
		&\leq &\frac{2\left \Vert u\right \Vert _{\infty }}{\Gamma \left( \varrho
			+1\right) }\left( \psi \left( t_{2}\right) -\psi \left( t_{1}\right) \right)
		^{\varrho }.
	\end{eqnarray*}
\end{proof}

\begin{theorem}
\cite{Smart} \label{T1}Let $\mathcal{B}_{r}$ be the closed ball of radius $%
r>0$, centred at zero, in a Banach space $X$ with $\Upsilon :\mathcal{B}
_{r}\rightarrow X$ a contraction and $\Upsilon (\partial \mathcal{\ B}
_{r})\subseteq \mathcal{B}_{r}$. Then, $\Upsilon $ has a unique fixed point
in $\mathcal{B}_{r}$.
\end{theorem}

\begin{theorem}
\cite{Smart} \label{T2}(Let $\mathcal{M}$ be a closed, convex, non-empty
subset of a Banach space $X\times X$. Suppose that $\mathbb{E}$ and $\mathbb{%
\ F}$ map $\mathcal{M}$ into $X$ and that

\begin{enumerate}
\item[(i)] $\mathbb{E}u+\mathbb{F}v\in \mathcal{M}$ for all $u,v\in \mathcal{%
M}$;

\item[(ii)] $\mathbb{E}$ is compact and continuous;

\item[(iii)] $\mathbb{F}$ is a contraction mapping.
\end{enumerate}
Then the operator equation $\mathbb{E}w+\mathbb{F}w=w$ has at least one
solution on $\mathcal{M}$.

\end{theorem}
\begin{definition}
\label{D1}The problem (\ref{EQ1}) is U-H stable if there exists a
real number $c_{f}>0$ such that for each $\epsilon >0$ and for each solution
$\tilde{u}\in \mathcal{C}\left[ a,T\right] $ of the inequality
\begin{equation}
\left\vert \left( ^{c}D_{a+,t}^{\varrho ,\psi }\right) \left(
^{c}D_{a+,t}^{\varsigma ,\psi }+\lambda \right) \left[ \tilde{u}\right]
-f(t,u(t),^{c}D_{a+,t}^{\delta ,\psi }\left[ \tilde{u}\right] )\right\vert
\leq \epsilon ,\ t\in \left[ a,T\right] ,  \label{EQ13}
\end{equation}%
there exists a solution $u\in \mathcal{C}\left[ a,T\right] $ of the problem (%
\ref{EQ1}) with
\begin{equation*}
\left\vert \tilde{u}(t)-u(t)\right\vert \leq \epsilon c_{f}.
\end{equation*}
\end{definition}

\begin{definition}
\label{D2}The problem (\ref{EQ1}) is generalized U-H stable if there
exists $\Phi \left( t\right) \in C\left(
\mathbb{R}
^{+},%
\mathbb{R}
^{+}\right)$,  $\Phi \left( 0\right) =0$ such that for each $\epsilon >0$ and
for each solution $\tilde{u}\in \mathcal{C}\left[ a,T\right] $ of inequality
(\ref{EQ13}), there exists a solution $u\in \mathcal{C}\left[ a,T\right] $
of problem (\ref{EQ1}) with
\begin{equation*}
\left\vert \tilde{u}(t)-u(t)\right\vert \leq \Phi \left( \epsilon \right) ,\
t\in \left[ a,T\right] ,
\end{equation*}%
where $\Phi \left( \epsilon \right) $ is only dependent on $\epsilon $.
\end{definition}

\begin{definition}
\label{D3}The problem (\ref{EQ1}) is U-H-R stable if there
exists a real number $c_{f}>0$ such that for each $\epsilon >0$ and for each
solution $\tilde{u}\in \mathcal{C}\left[ a,T\right] $ of the inequality
\begin{equation}
\left\vert \left( ^{c}D_{a+,t}^{\varrho ,\psi }\right) \left(
^{c}D_{a+,t}^{\varsigma ,\psi }+\lambda \right) \left[ \tilde{u}\right]
-f(t,u(t),^{c}D_{a+,t}^{\delta ,\psi }\left[ \tilde{u}\right] )\right\vert
\leq \epsilon \Phi \left( t\right) ,\ t\in \left[ a,T\right] ,  \label{EQ14}
\end{equation}%
there exists a solution $u\in \mathcal{C}\left[ a,T\right] $ of the problem (%
\ref{EQ1}) with
\begin{equation*}
\left\vert \tilde{u}(t)-u(t)\right\vert \leq \epsilon c_{f}\Phi \left(
t\right) .
\end{equation*}
\end{definition}

\begin{definition}
\label{D4}The problem (\ref{EQ1}) is generalized U-H-R stable
with respect to $\Phi $ if there exists $c_{f}>0$ such that for each
solution $\tilde{u}\in \mathcal{C}\left[ a,T\right] $ of the inequality
\begin{equation}
\left\vert ^{c}D_{a+,t}^{\varrho ,\psi }\left( ^{c}D_{a+,t}^{\varsigma ,\psi
}+\lambda \right) \left[ u\right] -f(t,u(t),^{c}D_{a+,t}^{\delta ,\psi }%
\left[ u\right] )\right\vert \leq \Phi \left( t\right) ,\ t\in \left[ a,T%
\right] ,  \label{EQ15}
\end{equation}%
there exists a solution $u\in \mathcal{C}\left[ a,T\right] $ of the problem (%
\ref{EQ1}) with
\begin{equation*}
\left\vert \tilde{u}(t)-u(t)\right\vert \leq c_{f}\Phi \left( t\right) .
\end{equation*}
\end{definition}

Below, we generalize Gronwall's inequality for $\psi $-fractional
derivative proved by Shi-you Lin in \cite{Shi}.

\begin{lemma}
\label{L6}Let $u,v$ be two integrable functions, with domain $[a,b]$ . Let $%
\psi \in \mathcal{C}^{1}[a,b]$ an increasing function such that $\psi
^{\prime }\left( t\right) \neq 0$, $\forall t\in \lbrack a,b]$. Assume that

\begin{enumerate}
\item[(i)] $u$ and $v$ are nonnegative;

\item[(ii)] The functions $\left( g_{i}\right) _{i=1,..n}\ $are the bounded
and monotonic increasing functions on $[a,b]$

\item[(iii)] The constants $\varrho _{i}>0\ \left( i=1,2,...,n\right) .\ $If
\begin{equation*}
u\left( t\right) \leq v\left( t\right) +\sum_{i=1}^{n}g_{i}\left( t\right)
\int_{a}^{t}\psi ^{\prime }\left( \tau \right) \left( \psi \left( t\right)
-\psi \left( \tau \right) \right) ^{\varrho _{i}-1}u\left( \tau \right)
\mathrm{d}\tau ,
\end{equation*}
then
\begin{equation*}
u\left( t\right) \leq v\left( t\right) +\sum_{k=1}^{\infty }\left(
\sum_{1^{\prime },2^{\prime },3^{\prime },...k^{\prime }=1}^{n}\frac{\prod
\limits_{i=1}^{k}\left( g_{i^{\prime }}\left( t\right) \Gamma \left( \varrho
_{i^{\prime }}\right) \right) }{\Gamma \left( \sum_{i=1}^{k}\varrho
_{i^{\prime }}\right) }\int_{a}^{t}\left[ \psi ^{\prime }\left( \tau \right)
\left( \psi \left( t\right) -\psi \left( \tau \right) \right)
^{\sum_{i=1}^{k}\varrho _{i^{\prime }}-1}\right] v\left( \tau \right) \mathrm{%
\ d}\tau \right) .
\end{equation*}
\end{enumerate}
\end{lemma}
\begin{remark}
\label{R1}For $n=2$ in the hypotheses of Lemma \ref%
{L6}.  Let $v(t)$ be a nondecreasing function on $a\leq t<T$. Then we have
\begin{equation*}
u\left( t\right) \leq v\left( t\right) \left[ E_{\varrho _{1},\psi }\left(
g_{1}\left( t\right) \Gamma \left( \varrho _{1}\right) \left( \psi \left(
t\right) \right) ^{\varrho _{1}}\right) +E_{\varrho _{2},\psi }\left(
g_{2}\left( t\right) \Gamma \left( \varrho _{2}\right) \left( \psi \left(
t\right) \right) ^{\varrho _{2}}\right) \right],
\end{equation*}
where $E_{\varrho _{1},\psi }$ is the Mittag--Leffler function defined below.
\end{remark}

\begin{definition}
\label{D5}\cite{GorenfloKMR} The Mittag--Leffler function is given by the
series
\begin{equation}
E_{\varrho }\left( z\right) =\sum_{k=0}^{\infty }\frac{z^{k}}{\Gamma \left(
\varrho k+1\right) },  \label{EQ16}
\end{equation}
where $Re\left( \varrho \right) >0$ and $\Gamma \left( z\right) $ is a gamma
function.
In particular, if $\varrho =1/2$ in (\ref{EQ16}) we have
\begin{equation*}
E_{1/2}\left( z\right) =\exp \left( z^{2}\right) \left[ 1+erf\left( z\right) %
\right] ,
\end{equation*}
where $erf\left( z\right) $ error function.
\end{definition}

The remaining portion of the paper, we make use of the next suppositions:
\begin{description}
\item[]
\begin{description}

\item[(A$_{1}$)] For each $t\in \left[ a,T\right] $, there exist a constant $%
L_{i}>0\ \left( i=1,2\right) $\ such that
\begin{equation*}
\left\vert f(t,u_{1},v_{1})-f(t,u_{2},v_{2})\right\vert \leq L_{1}\left\vert
u_{1}-v_{2}\right\vert +L_{2}\left\vert u_{1}-v_{2}\right\vert ,\text{ for
all }u_{i},v_{i}\in \mathbb{R}.
\end{equation*}

\item[(A$_{2}$)] There exists an increasing function $\chi \left( t\right)
\in \left( \mathcal{C}\left[ a,T\right] ,%
\mathbb{R}
^{+}\right) $, for any $t\in \left[ a,T\right] ,$
\begin{equation*}
\left\vert f(t,u,v)\right\vert \leq \chi \left( t\right) ,\text{ }u,v\in
\mathbb{R}.
\end{equation*}

\item[(A$_{3}$)] There exist a constant $L>0$ such that
\begin{equation*}
\left\vert f(t,u,v)\right\vert \leq L,\text{ for any }t\in \left[ a,T\right]
,\text{ }u,v\in \mathbb{R}.
\end{equation*}

\item[(A$_{4}$)] There exists an increasing function $\Phi \left( t\right)
\in \left( \mathcal{C}\left[ a,T\right] ,%
\mathbb{R}
^{+}\right) $ and there exists $l_{\Phi }>0$ such that for any $t\in \left[
a,T\right] ,$
\begin{equation*}
\left( J_{a+,t}^{\varrho ,\psi }\right) \left[ \Phi \right] \leq l_{\Phi
}\Phi \left( t\right) ,\ \varrho >0.
\end{equation*}
\end{description}
\end{description}
We adopt the following conventions:%
\begin{equation}
f_{u}(t)=f(t,u(t),^{c}D_{a+,t}^{\delta ,\psi }\left[ u\right] )\ \text{and }%
\mathcal{K}\left( t;a\right) =\psi \left( t\right) -\psi \left( a\right) .
\label{EQ17}
\end{equation}
\begin{equation}
\sigma _{11}=\left( J_{a_{+},\eta }^{\varsigma ,\psi }\right) \left[ 1\right]
\text{ and}\ \sigma _{12}=\left( J_{a_{+},\eta }^{\varsigma ,\psi }\right) \left[
\mathcal{K}\left( \tau ;a\right) \right] ,  \label{EQ18}
\end{equation}%
and
\begin{equation}
\sigma _{21}=\left( J_{a_{+},T}^{\varsigma ,\psi }\right) \left[ 1\right] -\mu
\left( J_{a_{+},\xi }^{\varsigma +\delta ,\psi }\right) \left[ 1\right] \text{
and}\ \sigma _{22}=\left( \left( J_{a_{+},T}^{\varsigma ,\psi }\right) \left[
\mathcal{K}\left( \tau ;a\right) \right] -\mu \left( J_{a_{+},\xi }^{\varsigma
+\delta ,\psi }\right) \left[ \mathcal{K}\left( \tau ;a\right) \right]
\right) .  \label{EQ19}
\end{equation}
Further, we assume
\begin{equation}
\left \vert \sigma _{11}\sigma _{22}-\sigma _{21}\sigma _{12}\right \vert
\neq 0.  \label{EQ20}
\end{equation}%
where $\sigma _{ij}$ are constants.

\section{Existence and uniqueness of solution}
In order to study the nonlinear FLE (\ref{EQ1}), We first consider the linear associated FLE and conclude solving it.

\subsection{Linear boundary problem}

\qquad The following Lemma regards a linear variant of problem%
\begin{equation}
\left\{
\begin{array}{l}
\left( ^{c}D_{a+,t}^{\varrho ,\psi }\right) \left( ^{c}D_{a+,t}^{\varsigma ,\psi
}+\lambda \right) \left[ u\right] =F(t),~t\in \left( a,T\right) ,\medskip \\
u(a)=0,~u(\eta )=0,~u(T)=\mu \left( J_{a+,\xi }^{\delta ,\psi }\right) \left[
u\right] ,~a<\eta <\xi <T,%
\end{array}%
\right.  \label{EQ21}
\end{equation}%
where $F\in \mathcal{C}([a,T],\mathbb{R}).$

\begin{lemma}
\label{L7}The unique solution of the $\psi $-Caputo linear problem (\ref%
{EQ21}) is given by the integral equation
\begin{align}
u(t)& =\text{$-\lambda \left( J_{a+,t}^{\varsigma ,\psi }\right) \left[ u\right]
+$}\left( \text{$J_{a+,t}^{\varrho +\varsigma ,\psi }$}\right) \text{$\left[ F %
\right] \medskip $}  \label{EQ22} \\
& \text{$+${\small $\frac{(\mathcal{K}\left( t;a\right) )^{\varsigma }(\mathcal{%
K }(t;\eta )}{\Gamma (\varsigma +2)\Delta }$}$\left \{ \left( J_{a+,T}^{\varrho
+\varsigma ,\psi }\right) \left[ F\right] -\lambda \left( J_{a+,T}^{\varsigma ,\psi
}\right) \left[ u\right] -\mu \left( J_{a+,\xi }^{\varrho +\varsigma +\delta
,\psi }\right) \left[ F\right] +\mu \lambda \left( J_{a+,\xi }^{\varsigma
+\delta ,\psi }\right) \left[ u\right] \right \} \medskip $}  \notag \\
\text{{}}& -\text{$\frac{(\mathcal{K}\left( t;a\right) )^{\varsigma }}{\Delta (
\mathcal{K}(\eta ;a))^{\varsigma }}$$\left( \text{$\frac{(\mathcal{K}\left(
T;a\right) )^{\varsigma }(\mathcal{K}\left( T;t\right) )}{\Gamma (\varsigma +2)}$}-
\text{$\frac{\mu (\mathcal{K}(\xi ;a))^{\varsigma +\delta }[(\varsigma +1)(\mathcal{%
K }(\xi ;t))-\delta (\mathcal{K}\left( t;a\right) )]}{\Gamma (\varsigma +\delta
+2)(\varsigma +1)}$}\right) \medskip $}  \notag \\
& \times \text{$\left \{ \left( J_{a+,\eta }^{\varrho +\varsigma ,\psi }\right) %
\left[ F\right] -\lambda \left( J_{a+,\eta }^{\varsigma ,\psi }\right) \left[ u %
\right] \right \} $}.  \notag
\end{align}%
where%
\begin{equation}
\Delta =\text{$\left[ \text{$\frac{(\mathcal{K}(T;a))^{\varsigma }\mathcal{K}
\left( T;\eta \right) }{\Gamma (\varsigma +2)}$}-\text{$\frac{\mu (\mathcal{K}
(\xi ;a))^{\varsigma +\delta }[(\varsigma +1)\mathcal{K}(\xi ;\eta )-\delta \mathcal{%
\ K}(\eta ;a)]}{\Gamma (\varsigma +\delta +2)(\varsigma +1)}$}\right] $}\neq 0.
\label{EQ23}
\end{equation}
\end{lemma}

\begin{proof}
	Applying $\left( J_{a+,t}^{\varrho ,\psi }\right) $ on both sides of (\ref%
	{EQ21}-a) , we have
	\begin{equation}
	\left( ^{c}D_{a+,t}^{\varsigma ,\psi }+\lambda \right) \left[ u\right] =\left(
	J_{a+,t}^{\varrho ,\psi }\right) \left[ F\right] +c_{1}+c_{2}(\psi (t)-\psi
	(a)),  \label{EQ24}
	\end{equation}%
	for $c_{1},~c_{2}\in \mathbb{R}.\ $\newline
	
	Now applying $\left( J_{a+,t}^{\varsigma ,\psi }\right) $ to both sides of (\ref%
	{EQ24}) , we get
	\begin{equation}
	u(t)=-\lambda \left( J_{a+,t}^{\varsigma ,\psi }\right) \left[ u\right] +\left(
	J_{a+,t}^{\varrho +\varsigma ,\psi }\right) \left[ F\right] +c_{1}\left(
	J_{a+,t}^{\varsigma ,\psi }\right) \left[ 1\right] +c_{2}\left( J_{a+,t}^{\varsigma
		,\psi }\right) \left[ \mathcal{K}\left( \tau ;a\right) \right] +c_{3},
	\label{EQ25}
	\end{equation}%
	where $c_{3}\in \mathbb{R}.\ $\newline
	
	Using the boundary conditions in (\ref{EQ21}-b), we obtain $%
	c_{3}:=c_{3}\left( F\right) =0$ and
	\begin{equation}
	J_{a+,t}^{\delta ,\psi }\left[ u\right] =-\lambda \left( J_{a+,t}^{\varsigma
		+\delta ,\psi }\right) \left[ u\right] +\left( J_{a+,t}^{\varrho +\varsigma
		+\delta ,\psi }\right) \left[ F\right] +c_{1}\left( J_{a+,t}^{\varsigma +\delta
		,\psi }\right) \left[ 1\right] +c_{2}\left( J_{a+,t}^{\varsigma +\delta ,\psi
	}\right) \left[ \mathcal{K}\left( \tau ;a\right) \right] .  \label{EQ26}
	\end{equation}
	Then we also get a system of linear equations with respect to $c_{1}$, $%
	c_{2} $ as follows%
	\begin{equation}
	\left(
	\begin{array}{cc}
	\sigma _{11} & \sigma _{12} \\
	\sigma _{21} & \sigma _{22}%
	\end{array}%
	\right) \left(
	\begin{array}{c}
	c_{1} \\
	c_{2}%
	\end{array}%
	\right) =\left(
	\begin{array}{c}
	b_{1} \\
	b_{2}%
	\end{array}%
	\right) ,  \label{EQ27}
	\end{equation}%
	where%
	\begin{equation}
	b_{1}=\lambda \left( J_{a_{+},\eta }^{\varsigma ,\psi }\right) \left[ u\right]
	-\left( J_{a_{+},\eta }^{\varrho +\varsigma ,\psi }\right) \left[ F\right]
	\label{EQ28}
	\end{equation}%
	and%
	\begin{equation}
	b_{2}=\lambda \left( \left( J_{a_{+},T}^{\varsigma ,\psi }\right) \left[ u\right]
	-\mu \left( J_{a_{+},\xi }^{\varsigma +\delta ,\psi }\right) \left[ u\right]
	\right) -\left( \left( J_{a_{+},T}^{\varrho +\varsigma ,\psi }\right) \left[ F%
	\right] -\mu \left( J_{a_{+},\xi }^{\varrho +\varsigma +\delta ,\psi }\right) %
	\left[ F\right] \right) .  \label{EQ29}
	\end{equation}
	We note%
	\begin{equation}
	\Delta \equiv \det \left( \sigma \right) =\left \vert \sigma _{11}\sigma
	_{22}-\sigma _{21}\sigma _{12}\right \vert .  \label{EQ30}
	\end{equation}
	Because the determinant of coefficient for $\Delta \neq 0.\ $Thus, we have%
	\begin{equation}
	c_{1}:=c_{1}\left( F\right) =\frac{\sigma _{22}b_{1}-\sigma _{12}b_{2}}{%
		\Delta }\text{ and }c_{2}:=c_{2}\left( F\right) =\frac{\sigma
		_{11}b_{2}-\sigma _{21}b_{1}}{\Delta }.  \label{EQ31}
	\end{equation}
	Substituting these values of $c_{1}$ and $c_{2}$ in (\ref{EQ26}), we finally
	obtain (\ref{EQ22}) as%
	\begin{equation}
	u\left( t\right) =-\lambda \left( J_{a+,t}^{\varsigma ,\psi }\right) \left[ u%
	\right] +\left( J_{a+,t}^{\varrho +\varsigma ,\psi }\right) \left[ F\right] +%
	\tfrac{\sigma _{22}b_{1}-\sigma _{12}b_{2}}{\Delta }\left( J_{a+,t}^{\varsigma
		,\psi }\right) \left[ 1\right] +\tfrac{\sigma _{11}b_{2}-\sigma _{21}b_{1}}{%
		\Delta }\left( J_{a+,t}^{\varsigma ,\psi }\right) \left[ \mathcal{K}\left( \tau
	;a\right) \right] ,  \label{EQ32}
	\end{equation}%
	i.e., the integral equation (\ref{EQ32})\ can be written as (\ref{EQ22}) and%
	\begin{equation}
	\left( J_{a+,t}^{\delta ,\psi }\right) \left[ u\right] =-\lambda \left(
	J_{a+,t}^{\varsigma +\delta ,\psi }\right) \left[ u\right] +\left(
	J_{a+,t}^{\varrho +\varsigma +\delta ,\psi }\right) \left[ F\right] +\tfrac{%
		\sigma _{22}b_{1}-\sigma _{12}b_{2}}{\Delta }\left( J_{a+,t}^{\varsigma +\delta
		,\psi }\right) \left[ 1\right] +\tfrac{\sigma _{11}b_{2}-\sigma _{21}b_{1}}{%
		\Delta }\left( J_{a+,t}^{\varsigma +\delta ,\psi }\right) \left[ \mathcal{K}%
	\left( \tau ;a\right) \right] .  \label{EQ33}
	\end{equation}
	Differentiating the above relations one time we obtain (\ref{EQ21}-a), also
	it is easy to get that the condition (\ref{EQ21}-b) is satisfied. The proof
	is complete.
\end{proof}

For convenience, we define the following functions%
\begin{equation}
d_{11}\left( t\right) =-\tfrac{1}{\Delta }\left( \sigma _{22}\left(
J_{a+,t}^{\varsigma ,\psi }\right) \left[ 1\right] -\sigma _{21}\left(
J_{a+,t}^{\varsigma ,\psi }\right) \left[ \mathcal{K}\left( \tau ;a\right) %
\right] \right) \text{, }d_{21}\left( t\right) =-d_{11}\left( t\right)
\label{EQ34}
\end{equation}%
and%
\begin{equation}
d_{12}\left( t\right) =\tfrac{1}{\Delta }\left( \sigma _{12}\left(
J_{a+,t}^{\varsigma ,\psi }\right) \left[ 1\right] -\sigma _{11}\left(
J_{a+,t}^{\varsigma ,\psi }\right) \left[ \mathcal{K}\left( \tau ;a\right) %
\right] \right) \text{, }d_{22}\left( t\right) =d_{12}\left( t\right) .
\label{EQ35}
\end{equation}

\subsection{Nonlinear problem}

\qquad Thanks to Lemma \ref{L7}, the following result is an immediate consequence.

\begin{lemma}
\label{L8} Let $\ \lambda >0.$ Then the problem (\ref{EQ1}) is equivalent to
the integral equation
\begin{equation}
u\left( t\right) =-\lambda \left( J_{a_{+},t}^{\varsigma ,\psi }\right) \left[ u %
\right] +\left( J_{a_{+},t}^{\varrho +\varsigma ,\psi }\right) \left[ f_{u}\right]
+\phi _{u}\left( f\right) ,  \label{EQ36}
\end{equation}
where
\begin{eqnarray}
\phi _{u}\left( f\right) &=&d_{11}\left( t\right) \left( J_{a_{+},\eta
}^{\varrho +\varsigma ,\psi }\right) \left[ f_{u}\right] +d_{12}\left( t\right)
\left( \left( J_{a_{+},T}^{\varrho +\varsigma ,\psi }\right) \left[ f_{u}\right]
-\mu \left( J_{a_{+},\xi }^{\varrho +\varsigma +\delta ,\psi }\right) \left[
f_{u} \right] \right) \medskip  \label{EQ37} \\
&&+\lambda d_{21}\left( t\right) \left( J_{a_{+},\eta }^{\varsigma ,\psi
}\right) \left[ u\right] -\lambda d_{22}\left( t\right) \left( \left(
J_{a_{+},T}^{\varsigma ,\psi }\right) \left[ u\right] -\mu \left( J_{a_{+},\xi
}^{\varsigma +\delta ,\psi }\right) \left[ u\right] \right)  \notag
\end{eqnarray}
and $d_{ij}$ are defined in (\ref{EQ34}) and (\ref{EQ35}).
\end{lemma}

From the expression of (\ref{EQ1}-a) and (\ref{EQ36}), we can see that if
all the conditions in Lemmas \ref{L7} and \ref{L8} are satisfied, the
solution is a $\mathcal{C}\left[ a,T\right] $ solution of the $\psi $-Caputo
fractional boundary value problem (\ref{EQ1}).

In order to lighten the statement of our result, we adopt the following
notation.%
\begin{equation}
\varsigma _{11}=\sup_{t\in \left[ a,T\right] }\left \vert \lambda \left(
J_{a_{+},t}^{\varsigma ,\psi }\right) \left[ 1\right] +\rho _{11}+L_{1}\left(
\left( J_{a_{+},t}^{\varrho +\varsigma ,\psi }\right) \left[ 1\right] +\rho
_{12}\right) \right \vert ,  \label{EQ38}
\end{equation}%
\begin{equation}
\varsigma _{12}=L_{2}\varsigma _{13}\text{ where }\varsigma _{13}=\sup_{t\in %
\left[ a,T\right] }\left \vert \left( J_{a_{+},t}^{\varrho +\varsigma ,\psi
}\right) \left[ 1\right] \right \vert +\rho _{12},  \label{EQ39}
\end{equation}%
\begin{equation}
\varsigma _{21}=\sup_{t\in \left[ a,T\right] }\left \vert \lambda \left(
J_{a+,t}^{\varsigma -\delta ,\psi }\right) \left[ 1\right] +\rho
_{21}+L_{1}\left( \left( J_{a+,t}^{\varrho +\varsigma -\delta ,\psi }\right) %
\left[ 1\right] +\rho _{22}\right) \right \vert ,  \label{EQ40}
\end{equation}%
\begin{equation}
\varsigma _{22}=L_{2}\varsigma _{23}\text{ where }\varsigma _{23}=\sup_{t\in %
\left[ a,T\right] }\left \vert \left( J_{a+,t}^{\varrho +\varsigma -\delta ,\psi
}\right) \left[ 1\right] \right \vert +\rho _{22},  \label{EQ41}
\end{equation}%
with%
\begin{equation}
\rho _{11}=\lambda \sup_{t\in \left[ a,T\right] }\left( \left \vert
d_{21}\left( t\right) \right \vert \left( J_{a_{+},\eta }^{\varsigma ,\psi
}\right) \left[ 1\right] +\left \vert d_{22}\left( t\right) \right \vert
\left( \left( J_{a_{+},T}^{\varsigma ,\psi }\right) \left[ 1\right] -\mu \left(
J_{a_{+},\xi }^{\varsigma +\delta ,\psi }\right) \left[ 1\right] \right) \right)
,  \label{EQ42}
\end{equation}%
\begin{equation}
\rho _{12}=\sup_{t\in \left[ a,T\right] }\left( \left \vert d_{11}\left(
t\right) \right \vert \left( J_{a_{+},\eta }^{\varrho +\varsigma ,\psi }\right) %
\left[ 1\right] +\left \vert d_{12}\left( t\right) \right \vert \left(
\left( J_{a_{+},T}^{\varrho +\varsigma ,\psi }\right) \left[ 1\right] -\mu \left(
J_{a_{+},\xi }^{\varrho +\varsigma +\delta ,\psi }\right) \left[ 1\right] \right)
\right) ,  \label{EQ43}
\end{equation}%
\begin{equation}
\rho _{21}=\lambda \sup_{t\in \left[ a,T\right] }\left( \left \vert \left(
^{c}D_{a+,t}^{\delta ,\psi }\right) \left[ d_{21}\right] \right \vert \left(
J_{a_{+},\eta }^{\varsigma ,\psi }\right) \left[ 1\right] +\left \vert \left(
^{c}D_{a+,t}^{\delta ,\psi }\right) \left[ d_{22}\right] \right \vert \left(
J_{a_{+},T}^{\varsigma ,\psi }\right) \left[ 1\right] -\mu \left( J_{a_{+},\xi
}^{\varsigma +\delta ,\psi }\right) \left[ 1\right] \right) ,  \label{EQ44}
\end{equation}%
and%
\begin{equation}
\rho _{22}=\sup_{t\in \left[ a,T\right] }\left( \left \vert \left(
^{c}D_{a+,t}^{\delta ,\psi }\right) \left[ d_{11}\right] \right \vert \left(
J_{a_{+},\eta }^{\varrho +\varsigma ,\psi }\right) \left[ 1\right] +\left \vert
\left( ^{c}D_{a+,t}^{\delta ,\psi }\right) \left[ d_{12}\right] \right \vert
\left( J_{a_{+},T}^{\varrho +\varsigma ,\psi }\right) \left[ 1\right] -\mu \left(
J_{a_{+},\xi }^{\varrho +\varsigma +\delta ,\psi }\right) \left[ 1\right] \right)
,  \label{EQ45}
\end{equation}

\qquad We are now in a position to establish the existence and
uniqueness results. Fixed point theorems are the main tool to prove this.

Let $\mathcal{C}=\mathcal{C}([a,T],\mathbb{R})$ be a Banach space of all
continuous functions defined on $[a,T]$ endowed with the usual supremum norm$%
.$ Define the space
\begin{equation}
E=\{u:u\in \mathcal{C}^{3}([a,T],\mathbb{R}),\ \left( ^{c}D_{a+,t}^{\delta
,\psi }\right) \left[ u\right] \in \mathcal{C}\},  \label{EQ46}
\end{equation}%
equipped with the norm
\begin{equation}
\left \Vert u\right \Vert _{E}=\max \left \{ \left \Vert u\right \Vert
_{\infty },\left \Vert \left( ^{c}D_{a+,t}^{\delta ,\psi }\right) \left[ u%
\right] \right \Vert _{\infty }\right \}.  \label{EQ47}
\end{equation}%
Then, we may conclude that $\left( E,\left \Vert .\right \Vert _{E}\right) $ is a Banach space.

To introduce a fixed point problem associated with (\ref{EQ36}) we consider
an integral operator $\Psi :E\rightarrow E$ defined by%
\begin{equation}
\left( \Psi u\right) \left( t\right) =-\lambda \left( J_{a_{+},t}^{\varsigma
,\psi }\right) \left[ u\right] +\left( J_{a_{+},t}^{\varrho +\varsigma ,\psi
}\right) \left[ f_{u}\right] +\phi _{u}\left( f\right) .  \label{EQ48}
\end{equation}

\begin{theorem}
\label{T3}Assume that $f:[a,T]\times
\mathbb{R}
\times
\mathbb{R}
\longrightarrow
\mathbb{R}
^{+}$ is a continuous function that satisfies (A$_{1}$). If we suppose that
\begin{equation}
0<\varsigma =\max \left\{ \varsigma _{11},\varsigma _{12},\varsigma
_{21},\varsigma _{22}\right\} <1,  \label{EQ49}
\end{equation}%
\ holds. Then the problem (\ref{EQ1}) has a unique solution on $E.$
\end{theorem}

\begin{proof}
	The proof will be given in two steps.

	\textbf{Step 1.} The operator $\Psi $ maps bounded sets into bounded sets in
	$E$.
	
	For our purpose, consider a function $u\in E$
	\ It is clear that $\Psi u\in E.$ Also by (%
	\ref{EQ12}),\ (\ref{EQ37})\ and ( \ref{EQ48}), we have
	\begin{equation}
	\left( ^{c}D_{a+,t}^{\delta ,\psi }\right) \left( \Psi u\right) =-\lambda
	\left( J_{a+,t}^{\varsigma -\delta ,\psi }\right) \left[ u\right] +\left(
	J_{a+,t}^{\varrho +\varsigma -\delta ,\psi }\right) \left[ f_{u}\right] +\left(
	^{c}D_{a+,t}^{\delta ,\psi }\right) \left[ \phi _{u}\left( f\right) \right] .
	\label{EQ50}
	\end{equation}
	Indeed, it is sufficient to prove that for any $r>0$, for each $u\in B_{r}=\left \{u\in E:\left\Vert u\right\Vert _{E}\leq r\right \},$
	we have $\left \Vert \Psi u\right \Vert _{E}\leq r.$
	
	Denoting
	\begin{equation}
	L_{0}=\sup_{t\in \left[ a,T\right] }\left \{ \left \vert f(t,0,0\right \vert
	:t\in \left[ a,T\right] \right \} <\infty \text{ and }L_{B}=L_{1}\sup_{t\in %
		\left[ a,T\right] }\left \vert u\left( t\right) \right \vert
	+L_{2}\sup_{t\in \left[ a,T\right] }\left \vert \left( ^{c}D_{a+,t}^{\delta
		,\psi }\right) \left[ u\right] \right \vert +L_{0}.  \label{EQ51}
	\end{equation}
	By (A$_{1}$) we have for each $t\in \left[ a,T\right] $
	\begin{equation}
	\left \vert f_{u}(t)\right \vert =\left \vert
	f_{u}(t)-f_{0}(t)+f_{0}(t)\right \vert \leq \left \vert
	f_{u}(t)-f_{0}(t)\right \vert +\left \vert f_{0}(t)\right \vert \leq L_{B},
	\label{EQ52}
	\end{equation}
	Firstly, we estimate $\left \vert \phi _{u}\left( f\right) \right \vert $ as
	follows%
	\begin{eqnarray}
	&&\left \vert \phi _{u}\left( f\right) \right \vert  \notag \\
    &=& \left \vert \lambda
	d_{21}\left( t\right) \left( J_{a_{+},\eta }^{\varsigma ,\psi }\right) \left[ u%
	\right] \right \vert +\left \vert d_{12}\left( t\right) \left( \left(
	J_{a_{+},T}^{\varrho +\varsigma ,\psi }\right) \left[ \left \vert
	f_{u}-f_{0}\right \vert +f_{0}\right] -\mu \left( J_{a_{+},\xi }^{\varrho
		+\varsigma +\delta ,\psi }\right) \left[ \left \vert f_{u}-f_{0}\right \vert
	+f_{0}\right] \right) \right \vert  \notag \\
	&&+\left \vert \lambda d_{22}\left( t\right) \left( \left(
	J_{a_{+},T}^{\varsigma ,\psi }\right) \left[ u\right] -\mu \left( J_{a_{+},\xi
	}^{\varsigma +\delta ,\psi }\right) \left[ u\right] \right) \right \vert +\left
	\vert d_{11}\left( t\right) \left( J_{a_{+},\eta }^{\varrho +\varsigma ,\psi
	}\right) \left[ \left \vert f_{u}-f_{0}\right \vert +f_{0}\right] \right
	\vert . \label{EQ53}
	\end{eqnarray}
	Then%
	\begin{eqnarray}
	\left \vert \phi _{u}\left( f\right) \right \vert &\leq &\left \vert
	d_{11}\left( t\right) \left( J_{a_{+},\eta }^{\varrho +\varsigma ,\psi }\right)
	\left[ L_{B}\right] \right \vert +\left \vert d_{12}\left( t\right) \left(
	\left( J_{a_{+},T}^{\varrho +\varsigma ,\psi }\right) \left[ L_{B}\right] -\mu
	\left( J_{a_{+},\xi }^{\varrho +\varsigma +\delta ,\psi }\right) \left[ L_{B}%
	\right] \right) \right \vert  \notag \\
	&&+\left \vert \lambda d_{21}\left( t\right) \left( J_{a_{+},\eta }^{\varsigma
		,\psi }\right) \left[ u\right] \right \vert +\left \vert \lambda
	d_{22}\left( t\right) \left( \left( J_{a_{+},T}^{\varsigma ,\psi }\right) \left[
	u\right] -\mu \left( J_{a_{+},\xi }^{\varsigma +\delta ,\psi }\right) \left[ u%
	\right] \right) \right \vert .  \label{EQ54}
	\end{eqnarray}
	Taking the maximum over $\left[ a,T\right] ,$ we get%
	\begin{equation}
	\sup_{t\in \left[ a,T\right] }\left \vert \phi _{u}\left( f\right) \right
	\vert \leq \rho _{11}\sup_{t\in \left[ a,T\right] }\left \vert u\left(
	t\right) \right \vert +\rho _{12}\left( L_{1}\sup_{t\in \left[ a,T\right]
	}\left \vert u\left( t\right) \right \vert +L_{2}\sup_{t\in \left[ a,T\right]
}\left \vert \left( ^{c}D_{a+,t}^{\delta ,\psi }\right) \left[ u\right]
\right \vert +L_{0}\right) ,  \label{EQ55}
\end{equation}%
where $\phi _{u}$, $d_{ij}\left( t\right) $ and $\rho _{ij}$ defined by (\ref%
{EQ37}), (\ref{EQ34}-\ref{EQ35}) and (\ref{EQ42}-\ref{EQ45}) respectively.

Using (\ref{EQ48}) and (\ref{EQ55}), we obtain%
\begin{equation}
\left \Vert \left( \Psi u\right) \right \Vert _{\infty }\leq \varsigma
_{11}\sup_{t\in \left[ a,T\right] }\left \vert u\left( t\right) \right \vert
+\varsigma _{12}\sup_{t\in \left[ a,T\right] }\left \vert \left(
^{c}D_{a+,t}^{\delta ,\psi }\right) \left[ u\right] \right \vert +\varsigma
_{13}L_{0},  \label{EQ56}
\end{equation}%
where $\varsigma _{ij}$ defined by (\ref{EQ38}-\ref{EQ39}).
 On the other hand%
\begin{eqnarray}
&&\left \vert \left( ^{c}D_{a+,t}^{\delta ,\psi }\right) \left[ \phi
_{u}\left( f\right) \right] \right \vert \notag \\
&=&\left( ^{c}D_{a+,t}^{\delta
	,\psi }\right) \left[ d_{11}\right] \left( J_{a_{+},\eta }^{\varrho +\varsigma
	,\psi }\right) \left[ f_{u}\right] +\left( ^{c}D_{a+,t}^{\delta ,\psi
}\right) \left[ d_{12}\right] \left( \left( J_{a_{+},T}^{\varrho +\varsigma ,\psi
}\right) \left[ f_{u}\right] -\mu \left( J_{a_{+},\xi }^{\varrho +\varsigma
+\delta ,\psi }\right) \left[ f_{u}\right] \right)  \notag \\
&&+\lambda \left( ^{c}D_{a+,t}^{\delta ,\psi }\right) \left[ d_{21}\right]
\left( J_{a_{+},\eta }^{\varsigma ,\psi }\right) \left[ u\right] -\lambda \left(
^{c}D_{a+,t}^{\delta ,\psi }\right) \left[ d_{22}\right] \left( \left(
J_{a_{+},T}^{\varsigma ,\psi }\right) \left[ u\right] -\mu \left( J_{a_{+},\xi
}^{\varsigma +\delta ,\psi }\right) \left[ u\right] \right). \label{EQ57}
\end{eqnarray}
Taking the maximum over $\left[ a,T\right] ,$ we get%
\begin{eqnarray}
&&\sup_{t\in \left[ a,T\right] }\left \vert \left( ^{c}D_{a+,t}^{\delta ,\psi
}\right) \left[ \phi _{u}\left( f\right) \right] \right \vert \notag \\
&\leq& \rho_{21}\sup_{t\in \left[ a,T\right] }\left \vert u\left( t\right) \right \vert
+\rho _{22}\left( L_{1}\sup_{t\in \left[ a,T\right] }\left \vert u\left(
t\right) \right \vert +L_{2}\sup_{t\in \left[ a,T\right] }\left \vert \left(
^{c}D_{a+,t}^{\delta ,\psi }\right) \left[ u\right] \right \vert
+L_{0}\right) .  \label{EQ58}
\end{eqnarray}
Using (\ref{EQ50}) and (\ref{EQ58}), we obtain%
\begin{equation}
\left \Vert \left( ^{c}D_{a+,t}^{\delta ,\psi }\right) \left( \Psi u\right)
\right \Vert _{\infty }\leq \varsigma _{21}\sup_{t\in \left[ a,T\right]
}\left \vert u\left( t\right) \right \vert +\varsigma _{22}\sup_{t\in \left[
a,T\right] }\left \vert \left( ^{c}D_{a+,t}^{\delta ,\psi }\right) \left[ u%
\right] \right \vert +\varsigma _{23}L_{0},  \label{EQ59}
\end{equation}
Consequently, by (\ref{EQ56}) and (\ref{EQ59}), we have
\begin{equation}
\left \Vert \left( \Psi u\right) \right \Vert _{E}\leq \varsigma \left \Vert
u\right \Vert _{E}+L_{0}\max \left \{ \varsigma _{13},\varsigma _{23}\right
\} \leq \varsigma r+\left( 1-\varsigma \right) r=r,  \label{EQ60}
\end{equation}%
where $\varsigma $ is defined by (\ref{EQ49}) and choose%
\begin{equation}
r>\frac{L_{0}\max \left \{ \varsigma _{13},\varsigma _{23}\right \} }{\left(
	1-\varsigma \right) },\ 0<\varsigma <1.  \label{EQ61}
\end{equation}
The continuity of the functional $f_{u}$ would imply the continuity of $%
\left( \Psi u\right) $ and $\left( ^{c}D_{a+,t}^{\delta ,\psi }\right)
\left( \Psi u\right) $ Hence, $\Psi $ maps bounded sets into
bounded sets in $E$.

\textbf{Step 2. }Now we show that $\Psi $ is a contraction. By
(A$_{1}$) and (\ref{EQ48}), for $u,v\in E$ and $t\in \left[ a,T\right] $, we
have%
\begin{equation}
\left \vert \left( \Psi u\right) \left( t\right) -\left( \Psi v\right)
\left( t\right) \right \vert \leq \lambda \left \vert \left(
J_{a_{+},t}^{\varsigma ,\psi }\right) \left[ u-v\right] \right \vert +\left
\vert \left( J_{a_{+},t}^{\varrho +\varsigma ,\psi }\right) \left[ f_{u}-f_{v}%
\right] \right \vert +\left \vert \phi _{u}\left( f\right) -\phi _{v}\left(
f\right) \right \vert ,  \label{EQ62}
\end{equation}%
where%
\begin{equation}
\left \vert \left( J_{a_{+},t}^{\varrho +\varsigma ,\psi }\right) \left[
f_{u}-f_{v}\right] \right \vert \leq \left \vert \left( J_{a_{+},t}^{\varrho
	+\varsigma ,\psi }\right) \left[ 1\right] \right \vert \left( L_{1}\left \vert
u-v\right \vert +L_{2}\left \vert \left( ^{c}D_{a+,t}^{\delta ,\psi }\right) %
\left[ u-v \right] \right \vert \right)  \label{EQ63}
\end{equation}%
and%
\begin{eqnarray}
&&\left \vert \phi _{u}\left( f\right) -\phi _{v}\left( f\right) \right \vert \hspace{1in}\hspace{1in}\qquad \hspace{1in} \notag \\
&&=d_{11}\left( t\right) \left( J_{a_{+},\eta }^{\varrho +\varsigma ,\psi
}\right) \left[ f_{u}-f_{v}\right] +d_{12}\left( t\right) \left( \left(
J_{a_{+},T}^{\varrho +\varsigma ,\psi }\right) \left[ f_{u}-f_{v}\right] -\mu
\left( J_{a_{+},\xi }^{\varrho +\varsigma +\delta ,\psi }\right) \left[
f_{u}-f_{v}\right] \right)  \notag \\
&&+\lambda d_{21}\left( t\right) \left( J_{a_{+},\eta }^{\varsigma ,\psi
}\right) \left[ u-v\right] -\lambda d_{22}\left( t\right) \left( \left(
J_{a_{+},T}^{\varsigma ,\psi }\right) \left[ u-v\right] -\mu \left( J_{a_{+},\xi
}^{\varsigma +\delta ,\psi }\right) \left[ u-v\right] \right),  \label{EQ64}
\end{eqnarray}%
for all $t\in \left[ a,T\right] $, which implies%
\begin{equation}
\left \vert \phi _{u}\left( f\right) -\phi _{v}\left( f\right) \right \vert
\leq \left( \rho _{11}+L_{1}\rho _{12}\right) \sup_{t\in \left[ a,T\right]
}\left \vert u\left( t\right) -v\left( t\right) \right \vert +\rho
_{12}L_{2}\sup_{t\in \left[ a,T\right] }\left \vert \left(
^{c}D_{a+,t}^{\delta ,\psi }\right) \left[ u-v %
\right] \right \vert .  \label{EQ65}
\end{equation}
Hence, we get%
\begin{eqnarray}
&&\left \vert \left( \Psi u\right) \left( t\right) -\left( \Psi v\right)
\left( t\right) \right \vert  \hspace{1in}\hspace{1in}\qquad \hspace{1in} \notag \\
 && \leq \lambda \left \vert \sup_{t\in \left[ a,T%
	\right] }\left( J_{a_{+},t}^{\varsigma ,\psi }\right) \left[ 1\right] \right
\vert \sup_{t\in \left[ a,T\right] }\left \vert u\left( t\right) -v\left(
t\right) \right \vert +\sup_{t\in \left[ a,T\right] }\left \vert \left( J_{a_{+},t}^{\varrho
	+\varsigma ,\psi }\right) \left[ 1\right] \right \vert \notag \\
&& \left( L_{1}\sup_{t\in %
	\left[ a,T\right] }\left \vert u\left( t\right) -v\left( t\right) \right
\vert +L_{2}\sup_{t\in \left[ a,T\right] }\left \vert \left(
^{c}D_{a+,t}^{\delta ,\psi }\right) \left[ u-v\right] \right \vert \right)
+\sup_{t\in \left[ a,T\right] }\left \vert \phi _{u}\left( f\right) -\phi
_{v}\left( f\right) \right \vert. \ \ \ \ \ \ \label{EQ66}
\end{eqnarray}%
Consequently,
\begin{equation}
\left \Vert \left( \Psi u\right) -\left( \Psi v\right) \right \Vert _{\infty
}\leq \varsigma _{11}\sup_{t\in \left[ a,T\right] }\left \vert u\left(
t\right) -v\left( t\right) \right \vert +\varsigma _{12}\sup_{t\in \left[ a,T%
	\right] }\left \vert \left( ^{c}D_{a+,t}^{\delta ,\psi }\right) \left[ u-v%
\right] \right \vert .  \label{EQ67}
\end{equation}
A similar argument shows that%
\begin{eqnarray}
&&\left \vert \left( ^{c}D_{a+,t}^{\delta ,\psi }\right) \left[ \left( \Psi
u\right) -\left( \Psi v\right) \right] \right \vert  \hspace{1in}\hspace{1in}\qquad \hspace{1in} \notag \\
&&=\lambda \left \vert
\left( J_{a+,t}^{\varsigma -\delta ,\psi }\right) \left[ u-v\right] \right \vert
+\left \vert \left( J_{a+,t}^{\varrho +\varsigma -\delta ,\psi }\right) \left[
f_{u}-f_{v}\right] \right \vert +\left \vert \left( ^{c}D_{a+,t}^{\delta
	,\psi }\right) \left[ \phi _{u}\left( f\right) -\phi _{v}\left( f\right) %
\right] \right \vert ,  \label{EQ68}
\end{eqnarray}%
where%
\begin{eqnarray}
&&\sup_{t\in \left[ a,T\right] }\left \vert \left( ^{c}D_{a+,t}^{\delta ,\psi
}\right) \left[ \phi _{u}\left( f\right) -\phi _{v}\left( f\right) \right]
\right \vert \hspace{1in}\hspace{1in}\qquad \hspace{1in} \notag \\
&&\leq \left( \rho _{21}+L_{1}\rho _{22}\right) \sup_{t\in \left[
	a,T\right] }\left \vert u\left( t\right) -v\left( t\right) \right \vert
+\rho _{22}L_{2}\sup_{t\in \left[ a,T\right] }\left \vert \left(
^{c}D_{a+,t}^{\delta ,\psi }\right) \left[ u-v\right] \right \vert +\rho
_{22}L_{0}.  \label{EQ69}
\end{eqnarray}
Combining\ (\ref{EQ68}) and (\ref{EQ69}), we obtain%
\begin{equation}
\left \Vert \left( ^{c}D_{a+,t}^{\delta ,\psi }\right) \left( \Psi u\right)
-\left( ^{c}D_{a+,t}^{\delta ,\psi }\right) \left( \Psi v\right) \right
\Vert _{\infty }\leq \varsigma _{21}\sup_{t\in \left[ a,T\right] }\left
\vert u\left( t\right) -v\left( t\right) \right \vert +\varsigma
_{22}\sup_{t\in \left[ a,T\right] }\left \vert \left( ^{c}D_{a+,t}^{\delta
	,\psi }\right) \left[ u-v\right] \right \vert .  \label{EQ70}
\end{equation}
Consequently, by (\ref{EQ67}) and (\ref{EQ70}), we have
\begin{equation}
\left \Vert \left( \Psi u\right) -\left( \Psi v\right) \right \Vert _{E}\leq
\varsigma \left \Vert u-v\right \Vert _{E}  \label{EQ71}
\end{equation}%
and choose $\varsigma =\max \left \{ \varsigma _{11},\varsigma
_{12},\varsigma _{21},\varsigma _{22}\right \} <1.\ $Hence, the operator $%
\Psi $ is a contraction,\ therefore $\Psi $ maps bounded sets into bounded
sets in $E$. Thus, the conclusion of the theorem follows by the contraction
mapping principle.
\end{proof}

For simplicity of presentation, we let%
\begin{equation}
\Lambda _{11}=\left( J_{a+,T}^{\varrho +\varsigma ,\psi }\right) \left[ 1\right]
+\left( J_{a+,T}^{\varrho +\varsigma -\delta ,\psi }\right) \left[ 1\right]
+\left( d_{11}\left( T\right) +\left( ^{c}D_{a+,T}^{\delta ,\psi }\right) %
\left[ d_{11}\right] \right) \left( J_{a_{+},\eta }^{\varrho +\varsigma ,\psi
}\right) \left[ 1\right] ,  \label{EQ72}
\end{equation}%
\begin{equation}
\Lambda _{12}=\left( d_{12}\left( T\right) +\left( ^{c}D_{a+,T}^{\delta
,\psi }\right) \left[ d_{12}\right] \right) \left( \left(
J_{a_{+},T}^{\varrho +\varsigma ,\psi }\right) \left[ 1\right] -\mu \left(
J_{a_{+},\xi }^{\varrho +\varsigma +\delta ,\psi }\right) \left[ 1\right] \right)
,  \label{EQ73}
\end{equation}

\begin{equation}
\Lambda _{21}=\left( J_{a+,T}^{\varsigma ,\psi }\right) \left[ 1\right]
+d_{21}\left( T\right) \left( J_{a_{+},\eta }^{\varsigma ,\psi }\right) \left[ 1%
\right] +d_{22}\left( T\right) \left( \left( J_{a_{+},T}^{\varsigma ,\psi
}\right) \left[ 1\right] -\mu \left( J_{a_{+},\xi }^{\varsigma +\delta ,\psi
}\right) \left[ 1\right] \right) ,  \label{EQ74}
\end{equation}%
\begin{equation}
\Lambda _{22}=\left( J_{a+,T}^{\varsigma -\delta ,\psi }\right) \left[ 1\right]
+\left( ^{c}D_{a+,T}^{\delta ,\psi }\right) \left[ d_{21}\right] \left(
J_{a_{+},\eta }^{\varsigma ,\psi }\right) \left[ 1\right] +\left(
^{c}D_{a+,T}^{\delta ,\psi }\right) \left[ d_{22}\right] \left( \left(
J_{a_{+},T}^{\varsigma ,\psi }\right) \left[ 1\right] -\mu \left( J_{a_{+},\xi
}^{\varsigma +\delta ,\psi }\right) \left[ 1\right] \right) .  \label{EQ75}
\end{equation}
We consider the space defined by (\ref{EQ46}) equipped with the norm
\begin{equation}
\left \Vert u\right \Vert _{E}=\left \Vert u\right \Vert _{\infty }+\left
\Vert \left( ^{c}D_{a+,t}^{\delta ,\psi }\right) \left[ u\right] \right
\Vert _{\infty }.  \label{EQ76}
\end{equation}
It is easy to know that $\left( E,\left \Vert .\right \Vert _{E}\right) $
is a Banach space with norm (\ref{EQ76}).
On this space, by virtue of Lemma \ref{L8}, we may define the operator $\Psi
:E\longrightarrow E$ by%
\begin{equation}
\left( \Psi u\right) \left( t\right) =\left( \Psi _{1}u\right) \left(
t\right) +\left( \Psi _{2}u\right) \left( t\right) =\left( -\lambda \left(
J_{a_{+},t}^{\varsigma ,\psi }\right) \left[ u\right] +\left(
J_{a_{+},t}^{\varrho +\varsigma ,\psi }\right) \left[ f_{u}\right] +\phi
_{u}\left( f\right) \right) ,  \label{EQ77}
\end{equation}%
where $\Psi _{1}$ and $\Psi _{2}$ the two operators defined on $B_{r}$ by%
\begin{equation}
\left( \Psi _{1}u\right) \left( t\right) =\left( J_{a+,t}^{\varrho +\varsigma
,\psi }\right) \left[ f_{u}\right] +d_{11}\left( t\right) \left(
J_{a_{+},\eta }^{\varrho +\varsigma ,\psi }\right) \left[ f_{u}\right]
+d_{12}\left( t\right) \left( \left( J_{a_{+},T}^{\varrho +\varsigma ,\psi
}\right) \left[ f_{u}\right] -\mu \left( J_{a_{+},\xi }^{\varrho +\varsigma
+\delta ,\psi }\right) \left[ f_{u}\right] \right)  \label{EQ78}
\end{equation}%
and%
\begin{equation}
\left( \Psi _{2}u\right) \left( t\right) =-\lambda \left( J_{a+,t}^{\varsigma
,\psi }\right) \left[ u\right] +\lambda d_{21}\left( t\right) \left(
J_{a_{+},\eta }^{\varsigma ,\psi }\right) \left[ u\right] -\lambda d_{22}\left(
t\right) \left( \left( J_{a_{+},T}^{\varsigma ,\psi }\right) \left[ u\right]
-\mu \left( J_{a_{+},\xi }^{\varsigma +\delta ,\psi }\right) \left[ u\right]
\right) ,  \label{EQ79}
\end{equation}%
where $d_{ij}\left( t\right) $ defined by (\ref{EQ34}) and (\ref{EQ35}).

Applying $\left( ^{c}D_{a+,t}^{\delta ,\psi }\right) $ on both sides of (\ref%
{EQ78}) and (\ref{EQ79}), we have%
\begin{eqnarray}
\left( ^{c}D_{a+,t}^{\delta ,\psi }\right) \left[ \Psi _{1}u\right]
&=&\left( J_{a+,t}^{\varrho +\varsigma -\delta ,\psi }\right) \left[ f_{u}\right]
+\left( ^{c}D_{a+,t}^{\delta ,\psi }\right) \left[ d_{11}\right] \left(
J_{a_{+},\eta }^{\varrho +\varsigma ,\psi }\right) \left[ f_{u}\right] \notag \\
&&+\left( ^{c}D_{a+,t}^{\delta ,\psi }\right) \left[ d_{12}\right] \left(
\left( J_{a_{+},T}^{\varrho +\varsigma ,\psi }\right) \left[ f_{u}\right] -\mu
\left( J_{a_{+},\xi }^{\varrho +\varsigma +\delta ,\psi }\right) \left[ f_{u}%
\right] \right)  \label{EQ80}
\end{eqnarray}%
and%
\begin{eqnarray}
\left( ^{c}D_{a+,t}^{\delta ,\psi }\right) \left[ \Psi _{2}u\right]
&=&-\lambda \left( J_{a+,t}^{\varsigma -\delta ,\psi }\right) \left[ u\right]
+\lambda \left( ^{c}D_{a+,t}^{\delta ,\psi }\right) \left[ d_{21}\right]
\left( J_{a_{+},\eta }^{\varsigma ,\psi }\right) \left[ u\right]  \notag \\
&&-\lambda \left( ^{c}D_{a+,t}^{\delta ,\psi }\right) \left[ d_{22}\right]
\left( \left( J_{a_{+},T}^{\varsigma ,\psi }\right) \left[ u\right] -\mu \left(
J_{a_{+},\xi }^{\varsigma +\delta ,\psi }\right) \left[ u\right] \right), \label{EQ81}
\end{eqnarray}
Thus, $\Psi $ is well-defined because $\Psi _{1}$ and $\Psi _{2}$ are
well-defined. The continuity of the functional $f_{u}$ confirms the
continuity of $\left( \Psi u\right) (t)$ and $\left( ^{c}D_{a+,t}^{\delta
,\psi }\right) \left[ \Psi u\right] (t)$, for each $t\in \left[ a,T\right] .$
Hence the operator $\Psi $ maps $E$ into itself.

In what follows, we utilize fixed point techniques to demonstrate the key results of this paper. In light of Lemma \ref{L8}, we turn problem (\ref{EQ36}) as
\begin{equation}
u=\Psi u,\ u\in E.  \label{EQ82}
\end{equation}
Notice that problem (\ref{EQ36}) has solutions if the operator $\Psi $ in (%
\ref{EQ82}) has fixed points. Doubtless that the fixed points of $\Psi $
are solutions of (\ref{EQ1}).
Consider the operator $\Psi :E\longrightarrow E.\ $For $u,v\in B_{r},$ we find that%
\begin{equation}
\left \Vert \Psi u\right \Vert _{E}=\left \Vert \Psi _{1}u\right \Vert
_{E}+\left \Vert \Psi _{2}u\right \Vert _{E}.  \label{EQ83}
\end{equation}

\begin{theorem}
\label{T4}Assume that $f:[a,T]\times \mathbb{R}\times\mathbb{R}\longrightarrow
\mathbb{R}^{+}$ is a continuous function and the assumption (A$_{3}$) holds. If
\begin{equation}
0<\lambda \left( \Lambda _{21}+\Lambda _{22}\right) <1,  \label{EQ84}
\end{equation}
then, the problem (\ref{EQ1}) has at least one fixed point on $[a,T]$.
\end{theorem}

\begin{proof}
	The proof will be specified in sundry steps:

\textbf{Step 1.} Firstly, we prove that, for any $u,v\in B_{r}$, $\Psi
_{1}u+\Psi _{2}v\in B_{r}$, it follows that
\begin{eqnarray*}
&&\left \Vert \left( \Psi _{1}u\right) \right \Vert _{E}=\left \Vert \left(
\Psi _{1}u\right) \right \Vert _{\infty }+\left \Vert \left(
^{c}D_{a+,t}^{\delta ,\psi }\right) \left( \Psi _{1}u\right) \right \Vert
_{\infty } \hspace{1in}\hspace{1in}\qquad \hspace{1in} \notag \\
&&\leq\left[ \left( J_{a+,T}^{\varrho +\varsigma ,\psi }\right) \left[ 1\right]
+d_{11}\left( T\right) \left( J_{a_{+},\eta }^{\varrho +\varsigma ,\psi }\right) %
\left[ 1\right] +d_{12}\left( T\right) \left( \left( J_{a_{+},T}^{\varrho
+\varsigma ,\psi }\right) \left[ 1\right] -\mu \left( J_{a_{+},\xi }^{\varrho
+\varsigma +\delta ,\psi }\right) \left[ 1\right] \right) \right] \left \Vert
f_{u}\right \Vert _{\infty } \notag \\
&&+\left[ \left( J_{a+,T}^{\varrho +\varsigma -\delta ,\psi }\right) \left[ 1 %
\right] +\left( ^{c}D_{a+,t}^{\delta ,\psi }\right) \left[ d_{11}\right]
\left( J_{a_{+},\eta }^{\varrho +\varsigma ,\psi }\right) \left[ 1\right] \right] \left \Vert
f_{u}\right \Vert _{\infty } \notag \\
&&+\left[\left(
^{c}D_{a+,t}^{\delta ,\psi }\right) \left[ d_{12}\right] \left( \left(
J_{a_{+},T}^{\varrho +\varsigma ,\psi }\right) \left[ 1\right] -\mu \left(
J_{a_{+},\xi }^{\varrho +\varsigma +\delta ,\psi }\right) \left[ 1\right] \right) %
\right] \left \Vert f_{u}\right \Vert _{\infty },
\end{eqnarray*}
\begin{eqnarray}
\label{EQ85}
\end{eqnarray}
we obtain
\begin{eqnarray}
&&\left \Vert \left( \Psi _{1}u\right) \right \Vert _{E}\times \left \Vert
f_{u}\right \Vert _{\infty}^{-1}  \hspace{1in}\hspace{1in}\qquad \hspace{1in} \notag \\
&&=\left( J_{a+,T}^{\varrho +\varsigma ,\psi }\right)
\left[ 1\right] +\left( J_{a+,T}^{\varrho +\varsigma -\delta ,\psi }\right) \left[
1\right] +\left( d_{11}\left( T\right) +\left( ^{c}D_{a+,t}^{\delta ,\psi
}\right) \left[ d_{11}\right] \right) \left( J_{a_{+},\eta }^{\varrho +\varsigma
,\psi }\right) \left[ 1\right]  \notag \\
&&+\left( d_{12}\left( T\right) +\left( ^{c}D_{a+,t}^{\delta ,\psi }\right)
\left[ d_{12}\right] \right) \left( \left( J_{a_{+},T}^{\varrho +\varsigma ,\psi
}\right) \left[ 1\right] -\mu \left( J_{a_{+},\xi }^{\varrho +\varsigma +\delta
,\psi }\right) \left[ 1\right] \right).  \label{EQ86}
\end{eqnarray}%
Then, we have%
\begin{equation}
\left \Vert \left( \Psi _{1}u\right) \right \Vert _{E}\leq \left( \Lambda
_{11}+\Lambda _{12}\right) \times \left \Vert f_{u}\right \Vert _{\infty},\
\Lambda _{11}+\Lambda _{12}<\infty ,  \label{EQ87}
\end{equation}%
which yields that $\Psi _{1}$ is bounded. On the opposite side
\begin{eqnarray}
&&\frac{1}{\lambda }\left \Vert \left( \Psi _{2}v\right) \right \Vert
_{E}\times \left \Vert v\right \Vert _{E}^{-1} \hspace{1in}\hspace{1in}%
\qquad \hspace{1in} \notag \\
&&\leq \left( J_{a+,T}^{\varsigma ,\psi }\right) +d_{21}\left( T\right) \left(
J_{a_{+},\eta }^{\varsigma ,\psi }\right) \left[ 1\right] +d_{22}\left( T\right)
\left( \left( J_{a_{+},T}^{\varsigma ,\psi }\right) \left[ 1\right] -\mu \left(
J_{a_{+},\xi }^{\varsigma +\delta ,\psi }\right) \left[ 1\right] \right) \notag \\
&&+\left( J_{a+,T}^{\varsigma -\delta ,\psi }\right) \left[ 1\right] +\left(
^{c}D_{a+,t}^{\delta ,\psi }\right) \left[ d_{21}\right] \left(
J_{a_{+},\eta }^{\varsigma ,\psi }\right) \left[ 1\right] +\left(
^{c}D_{a+,t}^{\delta ,\psi }\right) \left[ d_{22}\right] \left( \left(
J_{a_{+},T}^{\varsigma ,\psi }\right) \left[ 1\right] -\mu \left( J_{a_{+},\xi
}^{\varsigma +\delta ,\psi }\right) \left[ 1\right] \right), \notag \\
\end{eqnarray}%
\item which implies
\begin{equation}
\left \Vert \left( \Psi _{2}v\right) \right \Vert _{E}\leq \lambda \left(
\Lambda _{21}+\Lambda _{22}\right) \times \left \Vert v\right \Vert _{E}.
\label{EQ89}
\end{equation}
Then, from (\ref{EQ87}) and (\ref{EQ89}), it follows that%
\begin{equation}
\left \Vert \Psi \left( u,v\right) \right \Vert _{E}\leq \left( \Lambda
_{11}+\Lambda _{12}\right) \times \left \Vert f_{u}\right \Vert _{\infty}+\lambda
\left( \Lambda _{21}+\Lambda _{22}\right) \times \left \Vert v\right \Vert
_{E}.  \label{EQ90}
\end{equation}
By (A$_{3}$) and (\ref{EQ90}), we have that%
\begin{equation}
\left( \Lambda _{11}+\Lambda _{12}\right) \times \left \Vert f_{u}\right
\Vert _{\infty}+\lambda \left( \Lambda _{21}+\Lambda _{22}\right) \times \left
\Vert v\right \Vert _{E}\leq \left( \Lambda _{11}+\Lambda _{12}\right)
\times L+\lambda \left( \Lambda _{21}+\Lambda _{22}\right) \times r\leq r.
\label{EQ91}
\end{equation}%
Then%
\begin{equation}
r>\frac{\left( \Lambda _{11}+\Lambda _{12}\right) \times L}{1-\lambda \left(
\Lambda _{21}+\Lambda _{22}\right) },\ 0<\lambda \left( \Lambda
_{21}+\Lambda _{22}\right) <1,  \label{EQ92}
\end{equation}%
which concludes that $\Psi _{1}u+\Psi _{2}v\in B_{r}.$for all $u,v\in B_{r}.$

\textbf{Step 2.} Next, for $u,v\in B_{r},\ \Psi _{2}$ is a contraction. From (%
\ref{EQ79}) and (\ref{EQ81}), we have
\begin{equation}
\left \Vert \left( \Psi _{2}u\right) -\left( \Psi _{2}v\right) \right \Vert
_{E}=\left \Vert \left( \Psi _{2}u\right) -\left( \Psi _{2}v\right) \right
\Vert _{\infty }+\left \Vert \left( ^{c}D_{a+,t}^{\delta ,\psi }\right)
\left( \Psi _{2}u\right) -\left( ^{c}D_{a+,t}^{\delta ,\psi }\right) \left(
\Psi _{2}v\right) \right \Vert _{\infty },  \label{EQ93}
\end{equation}%
where%
\begin{eqnarray}
&&\left \Vert \left( \Psi _{2}u\right) -\left( \Psi _{2}v\right) \right
\Vert _{\infty }\leq  \label{EQ94} \\
&&\lambda \left( \left( J_{a+,t}^{\varsigma ,\psi }\right) \left[ 1\right]
+d_{21}\left( T\right) \left( J_{a_{+},\eta }^{\varsigma ,\psi }\right) \left[ 1%
\right] +d_{22}\left( T\right) \left( \left( J_{a_{+},T}^{\varsigma ,\psi
}\right) \left[ 1\right] -\mu \left( J_{a_{+},\xi }^{\varsigma +\delta ,\psi
}\right) \left[ 1\right] \right) \right) \left \Vert u-v\right \Vert
_{\infty }.  \notag
\end{eqnarray}%
From (\ref{EQ94}), we can write%
\begin{equation}
\left \Vert \left( \Psi _{2}u\right) -\left( \Psi _{2}v\right) \right \Vert
_{\infty }\leq \lambda \Lambda _{21}\times \left \Vert u-v\right \Vert
_{\infty }.  \label{EQ95}
\end{equation}%
On the other hand%
\begin{eqnarray*}
&&\left \Vert \left( ^{c}D_{a+,t}^{\delta ,\psi }\right) \left( \Psi
_{2}u\right) -\left( ^{c}D_{a+,t}^{\delta ,\psi }\right) \left( \Psi
_{2}v\right) \right \Vert _{\infty }\times \left \Vert u-v\right \Vert
_{\infty }^{-1}\leq \\ 
&&\lambda \left( \left( J_{a+,t}^{\varsigma -\delta ,\psi }\right) \left[ 1%
\right] +\left( ^{c}D_{a+,T}^{\delta ,\psi }\right) \left[ d_{21}\right]
\left( J_{a_{+},\eta }^{\varsigma ,\psi }\right) \left[ 1\right] +\left(
^{c}D_{a+,T}^{\delta ,\psi }\right) \left[ d_{22}\right] \left( \left(
J_{a_{+},T}^{\varsigma ,\psi }\right) \left[ 1\right] -\mu \left( J_{a_{+},\xi
}^{\varsigma +\delta ,\psi }\right) \left[ 1\right] \right) \right) ,  \notag
\end{eqnarray*}%
which yields%
\begin{equation}
\left \Vert \left( ^{c}D_{a+,t}^{\delta ,\psi }\right) \left( \Psi
_{2}u\right) -\left( ^{c}D_{a+,t}^{\delta ,\psi }\right) \left( \Psi
_{2}v\right) \right \Vert _{\infty }\leq \lambda \Lambda _{22}\times \left
\Vert u-v\right \Vert _{\infty }.  \label{EQ97}
\end{equation}%
So, using (\ref{EQ95}) and (\ref{EQ97}), it follows that%
\begin{equation}
\left \Vert \left( \Psi _{2}u\right) +\left( \Psi _{2}v\right) \right \Vert
_{E}\leq \lambda \left( \Lambda _{21}+\Lambda _{22}\right) \times \left
\Vert u-v\right \Vert _{E}.  \label{EQ98}
\end{equation}%
and choose $0<\lambda \left( \Lambda _{21}+\Lambda _{22}\right) <1.\ $Hence,
the operator $\Psi _{2}$ is a contraction.

\textbf{Step 3.} The continuity of $\Psi _{1}$ follows from that of $f_{u}.\ $%
Let $\left \{ u_{n}\right \} $ be a sequence such that $u_{n}\longrightarrow
u$ in $E$. Then for each $t\in \left[ a,T\right] $
\begin{eqnarray*}
\left \vert \left( \Psi _{1}u_{n}\right) \left( t\right) -\left( \Psi
_{1}u\right) \left( t\right) \right \vert &=&\left( J_{a+,t}^{\varrho +\varsigma
,\psi }\right) \left[ f_{u_{n}}-f_{u}\right] +d_{11}\left( t\right) \left(
J_{a_{+},\eta }^{\varrho +\varsigma ,\psi }\right) \left[ f_{u_{n}}-f_{u}\right] \\
&&+d_{12}\left( t\right) \left( \left( J_{a_{+},T}^{\varrho +\varsigma ,\psi
}\right) \left[ f_{u_{n}}-f_{u}\right] -\mu \left( J_{a_{+},\xi }^{\varrho
+\varsigma +\delta ,\psi }\right) \left[ f_{u_{n}}-f_{u}\right] \right) .  \notag
\end{eqnarray*}%
By last equality with equation (\ref{EQ78}), we can write%
\begin{eqnarray*}
&&\left \vert \left( \Psi _{1}u_{n}\right) \left( t\right) -\left( \Psi
_{1}u\right) \left( t\right) \right \vert \leq \\ 
&&\left[\left( J_{a+,t}^{\varrho +\varsigma ,\psi }\right) \left[ 1\right]
+d_{11}\left( t\right) \left( J_{a_{+},\eta }^{\varrho +\varsigma ,\psi }\right) %
\left[ 1\right] +d_{12}\left( t\right) \left( \left( J_{a_{+},T}^{\varrho
+\varsigma ,\psi }\right) \left[ 1\right] -\mu \left( J_{a_{+},\xi }^{\varrho
+\varsigma +\delta ,\psi }\right) \left[ 1\right] \right)\right] \sup_{t\in \left[ a,T %
\right] }\left \vert f_{u_{n}}-f_{u}\right \vert .  \notag
\end{eqnarray*}%
It follows from (\ref{EQ85}) that%
\begin{eqnarray}
&&\left \Vert \left( \Psi _{1}u_{n}\right) -\left( \Psi _{1}u\right) \right
\Vert _{E}=\left \Vert \left( \Psi _{1}u_{n}\right) -\left( \Psi
_{1}u\right) \right \Vert _{\infty }+\left \Vert \left( ^{c}D_{a+,t}^{\delta
,\psi }\right) \left( \left( \Psi _{1}u_{n}\right) -\left( \Psi _{1}u\right)
\right) \right \Vert _{\infty }\leq \notag \\ 
&&\left[ \left( J_{a+,T}^{\varrho +\varsigma ,\psi }\right) +d_{11}\left(
T\right) \left( J_{a_{+},\eta }^{\varrho +\varsigma ,\psi }\right) +d_{12}\left(
T\right) \left( \left( J_{a_{+},T}^{\varrho +\varsigma ,\psi }\right) -\mu \left(
J_{a_{+},\xi }^{\varrho +\varsigma +\delta ,\psi }\right) \right) \right] \left
\Vert f_{u_{n}}-f_{u}\right \Vert _{\infty } \notag \\
&&+\left[ \left( J_{a+,T}^{\varrho +\varsigma -\delta ,\psi }\right) +\left(
^{c}D_{a+,t}^{\delta ,\psi }\right) \left[ d_{11}\right] \left(
J_{a_{+},\eta }^{\varrho +\varsigma ,\psi }\right)\right. \notag \\
 &&+ \left. \left( ^{c}D_{a+,t}^{\delta
,\psi }\right) \left[ d_{12}\right] \left( \left( J_{a_{+},T}^{\varrho +\varsigma
,\psi }\right) -\mu \left( J_{a_{+},\xi }^{\varrho +\varsigma +\delta ,\psi
}\right) \right) \right] \left \Vert f_{u_{n}}-f_{u}\right \Vert _{\infty }.
\label{EQ102}
\end{eqnarray}%
By (\ref{EQ102}), we have%
\begin{eqnarray}
&&\left \Vert \left( \Psi _{1}u_{n}\right) -\left( \Psi _{1}u\right) \right
\Vert _{E}\left \Vert f_{u_{n}}-f_{u}\right \Vert _{\infty}^{-1}\leq \notag \\
&&\left( J_{a+,T}^{\varrho +\varsigma ,\psi }\left[ 1\right] \right) +\left(
J_{a+,T}^{\varrho +\varsigma -\delta ,\psi }\left[ 1\right] \right) +\left(
d_{11}\left( T\right) +\left( ^{c}D_{a+,t}^{\delta ,\psi }\right) \left[
d_{11}\right] \right) \left( J_{a_{+},\eta }^{\varrho +\varsigma ,\psi }\left[ 1%
\right] \right)  \notag \\
&&+\left( d_{12}\left( T\right) +\left( ^{c}D_{a+,t}^{\delta ,\psi }\right)
\left[ d_{12}\right] \right) \left( \left( J_{a_{+},T}^{\varrho +\varsigma ,\psi }%
\left[ 1\right] \right) -\mu \left( J_{a_{+},\xi }^{\varrho +\varsigma +\delta
,\psi }\left[ 1\right] \right) \right) . \label{EQ103}
\end{eqnarray}%
Consequently, by (\ref{EQ103}), we have%
\begin{equation}
\left \Vert \left( \Psi _{1}u_{n}\right) -\left( \Psi _{1}u\right) \right
\Vert _{\infty}\leq \left( \Lambda _{11}+\Lambda _{12}\right) \times \left \Vert
f_{u_{n}}-f_{u}\right \Vert _{\infty},\ \Lambda _{11}+\Lambda _{12}<\infty .
\label{EQ104}
\end{equation}%
Since $f_{u}$ is a continuous function, then by the Lebesgue dominated
convergence theorem which implies%
\begin{equation}
\left \Vert \left( \Psi _{1}u_{n}\right) -\left( \Psi _{1}u\right) \right
\Vert _{E}\longrightarrow 0\text{ as }n\longrightarrow \infty .
\label{EQ105}
\end{equation}
Furthermore, $\Psi _{1}$ is uniformly bounded on $B_{r}$ as $%
\left
\Vert \left( \Psi _{1}u\right) \right \Vert _{E}\leq \left( \Lambda
_{11}+\Lambda _{12}\right) \times \left \Vert f_{u}\right \Vert _{\infty},$ due to (\ref{EQ87}).

\textbf{Step 4.} Finally, we establish the compactness of $\Psi _{1}.$ Let $%
u,v\in B_{r}$, for $t_{1},t_{2}\in \left[ a,T\right] ,~t_{1}<t_{2},$ we have%
\begin{eqnarray}
&&\left \Vert \left( \Psi _{1}u\right) \left( t_{2}\right) -\left( \Psi
_{1}u\right) \left( t_{1}\right) \right \Vert _{\infty } \notag \\ 
&&\leq \left[ \left( J_{a+,t_{2}}^{\varrho +\varsigma ,\psi }\right) \left[ 1\right]
+d_{11}\left( t_{2}\right) \left( J_{a_{+},\eta }^{\varrho +\varsigma ,\psi
}\right) \left[ 1\right] +d_{12}\left( t_{2}\right) \left( \left(
J_{a_{+},T}^{\varrho +\varsigma ,\psi }\right) \left[ 1\right] -\mu \left(
J_{a_{+},\xi }^{\varrho +\varsigma +\delta ,\psi }\right) \left[ 1\right] \right) %
\right] \left \Vert f_{u}\right \Vert _{\infty }  \notag \\
&&-\left[ \left( J_{a+,t_{1}}^{\varrho +\varsigma ,\psi }\right) \left[ 1\right]
+d_{11}\left( t_{1}\right) \left( J_{a_{+},\eta }^{\varrho +\varsigma ,\psi
}\right) \left[ 1\right] +d_{12}\left( t_{1}\right) \left( \left(
J_{a_{+},T}^{\varrho +\varsigma ,\psi }\right) \left[ 1\right] -\mu \left(
J_{a_{+},\xi }^{\varrho +\varsigma +\delta ,\psi }\right) \right) \left[ 1\right] %
\right] \left \Vert f_{u}\right \Vert _{\infty }  \notag \\
&&\leq \left[\left( \left( J_{a+,t_{2}}^{\varrho +\varsigma ,\psi }\right) \left[ 1\right]
-\left( J_{a+,t_{1}}^{\varrho +\varsigma ,\psi }\right) \left[ 1\right] \right)
+\Lambda _{41}\left( J_{a_{+},\eta }^{\varrho +\varsigma ,\psi }\right) \left[ 1%
\right] \right. \notag \\
&&\left.+\Lambda _{42}\left( \left( J_{a_{+},T}^{\varrho +\varsigma ,\psi
}\right) \left[ 1\right] -\mu \left( J_{a_{+},\xi }^{\varrho +\varsigma +\delta
,\psi }\right) \left[ 1\right] \right)\right]\left \Vert f_{u}\right \Vert _{\infty }.  \label{EQ107}
\end{eqnarray}%
On the other hand%
\begin{eqnarray}
&&\left \Vert \left( ^{c}D_{a+,t}^{\delta ,\psi }\right) \left( \Psi
_{1}u\right) \left( t_{2}\right) -\left( ^{c}D_{a+,t}^{\delta ,\psi }\right)
\left( \Psi _{1}u\right) \left( t_{1}\right) \right \Vert _{\infty } \notag \\
&&\leq\left[ \left( J_{a+,t_{2}}^{\varrho +\varsigma -\delta ,\psi }\right) +\left(
^{c}D_{a+,t_{2}}^{\delta ,\psi }\right) \left[ d_{11}\right] \left(
J_{a_{+},\eta }^{\varrho +\varsigma ,\psi }\right) +\left(
^{c}D_{a+,t_{2}}^{\delta ,\psi }\right) \left[ d_{12}\right] \left( \left(
J_{a_{+},T}^{\varrho +\varsigma ,\psi }\right) -\mu \left( J_{a_{+},\xi }^{\varrho
+\varsigma +\delta ,\psi }\right) \right) \right] \left \Vert f_{u}\right \Vert
_{\infty }  \notag \\
&&-\left[ \left( J_{a+,t_{1}}^{\varrho +\varsigma -\delta ,\psi }\right) +\left(
^{c}D_{a+,t_{1}}^{\delta ,\psi }\right) \left[ d_{11}\right] \left(
J_{a_{+},\eta }^{\varrho +\varsigma ,\psi }\right) +\left(
^{c}D_{a+,t_{1}}^{\delta ,\psi }\right) \left[ d_{12}\right] \left( \left(
J_{a_{+},T}^{\varrho +\varsigma ,\psi }\right) -\mu \left( J_{a_{+},\xi }^{\varrho
+\varsigma +\delta ,\psi }\right) \right) \right] \left \Vert f_{u}\right \Vert
_{\infty }  \notag \\
&&\left[\left( \left( J_{a+,t_{2}}^{\varrho +\varsigma -\delta ,\psi }\right) -\left(
J_{a+,t_{1}}^{\varrho +\varsigma -\delta ,\psi }\right) \right) +\Lambda
_{43}\left( J_{a_{+},\eta }^{\varrho +\varsigma ,\psi }\right) +\Lambda
_{44}\left( \left( J_{a_{+},T}^{\varrho +\varsigma ,\psi }\right) -\mu \left(
J_{a_{+},\xi }^{\varrho +\varsigma +\delta ,\psi }\right) \right)\right] \left \Vert f_{u}\right \Vert
_{\infty }. \label{EQ109}
\end{eqnarray}%
Using (\ref{EQ107}) and ( \ref{EQ109}), we get%
\begin{eqnarray}
&&\left \Vert \left( \Psi _{1}u\right) \left( t_{2}\right) -\left( \Psi
_{1}u\right) \left( t_{1}\right) \right \Vert _{E} \notag \\
&&\leq\left[\left( \left( J_{a+,t_{2}}^{\varrho +\varsigma ,\psi }\right) \left[ 1\right]
-\left( J_{a+,t_{1}}^{\varrho +\varsigma ,\psi }\right) \left[ 1\right] \right)
+\left( \left( J_{a+,t_{2}}^{\varrho +\varsigma -\delta ,\psi }\right) -\left(
J_{a+,t_{1}}^{\varrho +\varsigma -\delta ,\psi }\right) \right) \right.  \notag \\
&&\left.+\left( \Lambda _{41}+\Lambda _{43}\right) \left( J_{a_{+},\eta }^{\varrho
+\varsigma ,\psi }\right) \left[ 1\right] +\left( \Lambda _{42}+\Lambda
_{44}\right) \left( \left( J_{a_{+},T}^{\varrho +\varsigma ,\psi }\right) \left[ 1%
\right] -\mu \left( J_{a_{+},\xi }^{\varrho +\varsigma +\delta ,\psi }\right) %
\left[ 1\right] \right) \right]\left \Vert f_{u}\right \Vert
_{\infty },  \notag
\end{eqnarray}%
where%
\begin{equation*}
\Lambda _{41}=d_{11}\left( t_{2}\right) -d_{11}\left( t_{1}\right) \text{
and }\Lambda _{42}=d_{12}\left( t_{2}\right) -d_{12}\left( t_{1}\right)
\end{equation*}%
and%
\begin{equation*}
\Lambda _{43}=\left( ^{c}D_{a+,t_{21}}^{\delta ,\psi }\right) \left[ d_{11}%
\right] -\left( ^{c}D_{a+,t_{1}}^{\delta ,\psi }\right) \left[ d_{11}\right]
\text{ and }\Lambda _{44}=\left( ^{c}D_{a+,t_{2}}^{\delta ,\psi }\right) %
\left[ d_{12}\right] -\left( ^{c}D_{a+,t_{1}}^{\delta ,\psi }\right) \left[
d_{11}\right] .
\end{equation*}%
Consequently, we have%
\begin{equation}
\left \Vert \left( \Psi _{1}u\right) \left( t_{2}\right) -\left( \Psi
_{1}u\right) \left( t_{1}\right) \right \Vert \times \left \Vert f_{u}\right
\Vert ^{-1}\longrightarrow 0\text{ as }t_{1}\rightarrow t_{2}.  \label{EQ113}
\end{equation}%
Thus, $\Psi _{1}$ is relatively compact on $B_{r}$.
Hence, by the Arzela-Ascoli Theorem, $\Psi _{1}$ is completely
continuous on $B_{r}$. Therefore, according to Theorem \ref{T2}, the problem (\ref{EQ1}%
) has at least one solution on $B_{r}$. This completes the proof.

\end{proof}

\section{Stability of solutions}

\qquad Hereafter, we discuss the U-H and U-H-R stability of solutions of the FLE (\ref%
{EQ1}). In the proofs of Theorems \ref{T5} and \ref{T8}, we use  integration by parts in
the settings of $\psi $-fractional operators. Denoting%
\begin{equation}
\varphi _{1}\left( t\right) =\left( J_{a_{+},t}^{\varsigma ,\psi }\right) \left[
1\right] \text{ and }\varphi _{2}\left( t\right) =\left( J_{a_{+},t}^{\varsigma
,\psi }\right) \left[ \mathcal{K}\left( \tau ;a\right) \right] .
\label{EQ114}
\end{equation}

\begin{remark}
\label{R2}For every $\epsilon >0,$ a function $\tilde{u}\in \mathcal{C}$ is
a solution of of the inequality
\begin{equation}
\left \vert \left( ^{c}D_{a+,t}^{\varrho ,\psi }\right) \left( ^{c}D_{a+,t
}^{\varsigma ,\psi }+\lambda \right) \left[ \tilde{u}\right] -f(t,\tilde{u}
(t),^{c}D_{a+,t}^{\delta ,\psi }\left[ \tilde{u}\right] )\right \vert \leq
\epsilon \Phi \left( t\right) ,\ t\in \left[ a,T\right] ,  \label{EQ115}
\end{equation}
where $\Phi \left( t\right) \geq 0$ if and only if there exists a function $%
g\in \mathcal{C},$ (which depends on $\tilde{u}$) such that
\begin{itemize}
  \item [(\textbf{i})] $\left \vert g\left( t\right) \right \vert \leq \epsilon \Phi
\left( t\right) ,\ \forall t\in \left[ a,T\right] .\ $
  \item [(\textbf{ii})] $\left(
^{c}D_{a+,t}^{\varrho ,\psi }\right) \left( ^{c}D_{a+,t }^{\varsigma ,\psi
}+\lambda \right) \left[ \tilde{u}\right] =f(t,\tilde{u}(t),^{c}D_{a+,t}^{
\delta ,\psi }\left[ \tilde{u}\right] )+g\left( t\right) .$
\end{itemize}

\end{remark}

\begin{lemma}
\label{L9}If $\tilde{u}\in \mathcal{C}$ is a solution of the inequation (\ref%
{EQ115}) then $\tilde{u}$ is a solution of the following integral inequation
\begin{equation}
\left \vert \tilde{u}(t)-\left( -\lambda \left( J_{a_{+},t}^{\varsigma ,\psi
}\right) \left[ \tilde{u}\right] +\left( J_{a_{+},t}^{\varrho +\varsigma ,\psi
}\right) \left[ f_{\tilde{u}}\right] +\phi _{\tilde{u}}\left( f\right)
\right) \right \vert \leq C_{\Phi }\left( t\right) ,  \label{EQ116}
\end{equation}%
where%
\begin{equation}
C_{\Phi }\left( t\right) =\epsilon \left( J_{a_{+},t}^{\varrho +\varsigma ,\psi
}\right) \left[ \Phi \right] +c_{1}\left( \epsilon \Phi \right) \varphi
_{1}\left( t\right) +c_{2}\left( \epsilon \Phi \right) \varphi _{2}\left(
t\right) +c_{3}\left( \epsilon \Phi \right) ,  \label{EQ117}
\end{equation}%
where $c_{1}\left( \epsilon \Phi \right) ,\ c_{2}\left( \epsilon \Phi
\right) ,\ c_{3}\left( \epsilon \Phi \right) $ are real constants with $f_{
\tilde{u}}=\Phi $ and $C_{\Phi }$ is independent of $\tilde{u}\left(
t\right) $ and $f_{\tilde{u}}.$
\end{lemma}

\begin{proof}
	Let $\tilde{u}\in \mathcal{C}$ be a solution of the inequality (\ref{EQ115}%
	). Then by Remark \ref{R2}-ii, we have that%
	\begin{equation}
	\tilde{u}\left( t\right) =-\lambda \left( J_{a_{+},t}^{\varsigma ,\psi }\right) %
	\left[ \tilde{u}\right] +\left( J_{a_{+},t}^{\varrho +\varsigma ,\psi }\right) %
	\left[ f_{\tilde{u}}+g\right] +\phi _{\tilde{u}}\left( f_{\tilde{u}%
	}+g\right) ,  \label{EQ118}
	\end{equation}%
	where%
	\begin{equation}
	\phi _{\tilde{u}}\left( f_{\tilde{u}}+g\right) =c_{1}\left( f_{\tilde{u}%
	}+g\right) \varphi _{1}\left( t\right) +c_{2}\left( f_{\tilde{u}}+g\right)
	\varphi _{2}\left( t\right) +c_{3}\left( f_{\tilde{u}}+g\right) ,
	\label{EQ119}
	\end{equation}%
	with%
	\begin{equation}
	c_{j}\left( f_{\tilde{u}}+g\right) =c_{j}\left( f_{\tilde{u}}\right)
	+c_{j}\left( g\right) ,\ j=1,2,3.  \label{EQ120}
	\end{equation}
	In view of (A$_{1}$) and (\ref{EQ116}), we obtain%
	\begin{eqnarray}
	&&\left \vert \tilde{u}\left( t\right) -\left( -\lambda \left(
	J_{a_{+},t}^{\varsigma ,\psi }\right) \left[ \tilde{u}\right] +\left(
	J_{a_{+},t}^{\varrho +\varsigma ,\psi }\right) \left[ f_{\tilde{u}}\right] +\phi
	_{\tilde{u}}\left( f\right) \right) \right \vert \hspace{1in} \notag \\  
	&&=\left \vert \left( J_{a_{+},t}^{\varrho +\varsigma ,\psi }\right) \left[ g%
	\right] +c_{1}\left( g\right) \varphi _{1}\left( t\right) +c_{2}\left(
	g\right) \varphi _{2}\left( t\right) +c_{3}\left( g\right) \right \vert
	\notag \\
	&&\leq \left \vert \left( J_{a_{+},t}^{\varrho +\varsigma ,\psi }\right) \left[
	\epsilon \Phi \right] +c_{1}\left( \epsilon \Phi \right) \varphi _{1}\left(
	t\right) +c_{2}\left( \epsilon \Phi \right) \varphi _{2}\left( t\right)
	+c_{3}\left( \epsilon \Phi \right) \right \vert =C_{\Phi }\left( t\right).
	\label{EQ121}
	\end{eqnarray}
\end{proof}
As an outcome of Lemma \ref{L9} we have the following result:

\begin{corollary}
\label{C1} Assume that $f_{\tilde{u}}$ is a continuous function that
satisfies (A$_{1}$). If $\tilde{u}\in \mathcal{C}$ is a solution of the
inequality
\begin{equation}
\left \vert ^{c}D_{a+,t}^{\varrho ,\psi }\left( ^{c}D_{a+,t}^{\varsigma ,\psi
}+\lambda \right) u(t)-f(t,u(t),^{c}D_{a+,t}^{\delta ,\psi }\left[ u\right]
(t))\right \vert \leq \epsilon ,\ t\in \left[ a,T\right] ,  \label{EQ122}
\end{equation}
then $\tilde{u}$ is a solution of the following integral inequality
\begin{equation}
\left \vert \tilde{u}(t)-\left( -\lambda \left( J_{a_{+},t}^{\varsigma ,\psi
}\right) \left[ \tilde{u}\right] +\left( J_{a_{+},t}^{\varrho +\varsigma ,\psi
}\right) \left[ f_{\tilde{u}}\right] +\phi _{\tilde{u}}\left( f\right)
\right) \right \vert \leq C_{\epsilon },  \label{EQ123}
\end{equation}%
with%
\begin{equation}
\tilde{u}(a)=0,~\tilde{u}(\eta )=0,~\tilde{u}(T)=\mu \left( J_{a+,\xi
}^{\delta ,\psi }\right) \left[ \tilde{u}\right] ,~a<\eta <\xi <T,\ 0<\mu ,
\label{EQ124}
\end{equation}%
where%
\begin{equation}
C_{\epsilon }=\epsilon \varsigma _{13},  \label{EQ125}
\end{equation}%
where $\varsigma _{13}$ is given by (\ref{EQ39}).
\end{corollary}

\begin{proof}
	By Remark \ref{R2}-ii, (\ref{EQ118}), and by using (\ref{EQ37}) with the
	conditions (\ref{EQ124}), we have%
	\begin{eqnarray}
	\phi _{\tilde{u}}\left( f+g\right) &=&d_{11}\left( t\right) \left(
	J_{a_{+},\eta }^{\varrho +\varsigma ,\psi }\right) \left[ f_{\tilde{u}}\right]
	+d_{12}\left( t\right) \left( \left( J_{a_{+},T}^{\varrho +\varsigma ,\psi
	}\right) \left[ f_{\tilde{u}}\right] -\mu \left( J_{a_{+},\xi }^{\varrho
	+\varsigma +\delta ,\psi }\right) \left[ f_{\tilde{u}}\right] \right) \notag \\
&&+\lambda d_{21}\left( t\right) \left( J_{a_{+},\eta }^{\varsigma ,\psi
}\right) \left[ \tilde{u}\right] -\lambda d_{22}\left( t\right) \left(
\left( J_{a_{+},T}^{\varsigma ,\psi }\right) \left[ \tilde{u}\right] -\mu \left(
J_{a_{+},\xi }^{\varsigma +\delta ,\psi }\right) \left[ \tilde{u}\right] \right)
\notag \\
&&+d_{11}\left( t\right) \left( J_{a_{+},\eta }^{\varrho +\varsigma ,\psi
}\right) \left[ g\right] +d_{12}\left( t\right) \left( \left(
J_{a_{+},T}^{\varrho +\varsigma ,\psi }\right) \left[ g\right] -\mu \left(
J_{a_{+},\xi }^{\varrho +\varsigma +\delta ,\psi }\right) \left[ g\right] \right).
\label{EQ126}
\end{eqnarray}
The solution of the problem (\ref{EQ118}) is given by%
\begin{eqnarray}
&&\left \vert \tilde{u}\left( t\right) -\left( -\lambda \left(
J_{a_{+},t}^{\varsigma ,\psi }\right) \left[ \tilde{u}\right] +\left(
J_{a_{+},t}^{\varrho +\varsigma ,\psi }\right) \left[ f_{\tilde{u}}\right] +\phi
_{\tilde{u}}\left( f\right) \right) \right \vert  \hspace{1in} \notag \\ 
&&\leq \left \vert \left( J_{a_{+},t}^{\varrho +\varsigma ,\psi }\right) \left[ g%
\right] +d_{11}\left( t\right) \left( J_{a_{+},\eta }^{\varrho +\varsigma ,\psi
}\right) \left[ g\right] +d_{12}\left( t\right) \left( \left(
J_{a_{+},T}^{\varrho +\varsigma ,\psi }\right) \left[ g\right] -\mu \left(
J_{a_{+},\xi }^{\varrho +\varsigma +\delta ,\psi }\right) \left[ g\right] \right)
\right \vert,  \label{EQ127}
\end{eqnarray}%
which implies that%
\begin{equation}
\left \vert \tilde{u}\left( t\right) -\left( -\lambda \left(
J_{a_{+},t}^{\varsigma ,\psi }\right) \left[ \tilde{u}\right] +\left(
J_{a_{+},t}^{\varrho +\varsigma ,\psi }\right) \left[ f_{\tilde{u}}\right] +\phi
_{\tilde{u}}\left( f\right) \right) \right \vert \leq \Phi _{\epsilon }\left(
t\right) ,  \label{EQ128}
\end{equation}%
where%
\begin{equation}
\Phi _{\epsilon }\left( t\right) =\left( J_{a_{+},t}^{\varrho +\varsigma ,\psi
}\right) \left[ \epsilon \right] +c_{1}\left( \epsilon \right) \varphi
_{1}\left( t\right) +c_{2}\left( \epsilon \right) \varphi _{2}\left(
t\right) +c_{3}\left( \epsilon \right)   \label{EQ130}
\end{equation}%
with%
\begin{equation}
C_{\epsilon }\equiv \sup_{t\in \left[ a,T\right] }\left \vert \Phi _{\epsilon
}\left( t\right) \right \vert =\epsilon \varsigma _{13},  \label{EQ131}
\end{equation}%
which is the desired inequality (\ref{EQ123}).
\end{proof}

This corollary is obtained from Lemma \ref{L9} by setting $\Phi \left(
t\right) =1$, for all $\ t\in \left[ a,T\right],$ with (\ref{EQ124}).

\begin{theorem}
\label{T5}Assume that $f_{\tilde{u}}$ is a continuous function that
satisfies (A$_{1}$) and (A$_{4}$). The equation (\ref{EQ1}-a) is
H-U-R stable with respect to $\Phi \ $if there exists a real
number $l_{\Phi }>0$ such that for each $\epsilon >0$ and for each solution $%
\tilde{u}\in \mathcal{C}^{3}\left( \left[ a,T\right] ,
\mathbb{R}
\right) $ of the inequality (\ref{EQ115})$,\ $there exists a solution $%
u^{\ast }\in \mathcal{C}^{3}\left( \left[ a,T\right] ,
\mathbb{R}
\right) $ of (\ref{EQ1}-a) with
\begin{equation}
\left \vert \tilde{u}(t)-u^{\ast }(t)\right \vert \leq \epsilon l_{\Phi
}\Phi \left( t\right) .  \label{EQ132}
\end{equation}
\end{theorem}

\begin{proof}
	Using (\ref{EQ115}) and (\ref{EQ1}), we obtain%
	\begin{eqnarray}
	\tilde{u}\left( t\right) &=&-\lambda \left( J_{a_{+},t}^{\varsigma ,\psi
	}\right) \left[ \tilde{u}\right] +\left( J_{a_{+},t}^{\varrho +\varsigma ,\psi
}\right) \left[ f_{\tilde{u}}+g\right] +c_{1}\left( f_{\tilde{u}}+g\right)
\varphi _{1}\left( t\right) +c_{2}\left( f_{\tilde{u}}+g\right) \varphi
_{2}\left( t\right) +c_{3}\left( f_{\tilde{u}}+g\right)  \notag \\
&=&\theta _{\tilde{u}}\left( t,f_{\tilde{u}}+g\right) +c_{1}\left(
f+g\right) \varphi _{1}\left( t\right) +c_{2}\left( f+g\right) \varphi
_{2}\left( t\right) +c_{3}\left( f+g\right)  \label{EQ133}
\end{eqnarray}%
and%
\begin{eqnarray}
u^{\ast }\left( t\right) &=&-\lambda \left( J_{a_{+},t}^{\varsigma ,\psi
}\right) \left[ u^{\ast }\right] +\left( J_{a_{+},t}^{\varrho +\varsigma ,\psi
}\right) \left[ f_{u^{\ast }}\right] +c_{1}^{\prime }\left( f_{u^{\ast
}}\right) \varphi _{1}\left( t\right) +c_{2}^{\prime }\left( f_{u^{\ast
}}\right) \varphi _{2}\left( t\right) +c_{3}^{\prime }\left( f_{u^{\ast
}}\right) \notag \\
&=&\theta _{u^{\ast }}\left( t,f\right) +c_{1}^{\prime }\left( f_{u^{\ast
	}}\right) \varphi _{1}\left( t\right) +c_{1}^{\prime }\left( f_{u^{\ast
}}\right) \varphi _{1}\left( t\right) +c_{2}^{\prime }\left( f_{u^{\ast
}}\right) \varphi _{2}\left( t\right) +c_{3}^{\prime }\left( f_{u^{\ast
}}\right) ,   \label{EQ134}
\end{eqnarray}%
where%
\begin{equation}
\theta _{\tilde{u}}\left( t,f\right) =-\lambda \left( J_{a_{+},t}^{\varsigma
	,\psi }\right) \left[ \tilde{u}\right] +\left( J_{a_{+},t}^{\varrho +\varsigma
	,\psi }\right) \left[ f_{\tilde{u}}\right] .  \label{EQ135}
\end{equation}
By using (\ref{EQ132}) and (\ref{EQ133}), we have the following inequalities%
\begin{eqnarray}
\left \vert \tilde{u}(t)-u^{\ast }(t)\right \vert &\leq &\left \vert \tilde{u%
}\left( t\right) -\left( \theta _{u^{\ast }}\left( t,f\right) +c_{1}^{\prime
}\left( f_{u^{\ast }}\right) \varphi _{1}\left( t\right) +c_{2}^{\prime
}\left( f_{u^{\ast }}\right) \varphi _{2}\left( t\right) +c_{3}^{\prime
}\left( f_{u^{\ast }}\right) \right) \right \vert \medskip  \notag \\
&\leq &\left \vert \tilde{u}\left( t\right) -\left( \theta _{\tilde{u}%
}\left( t,f\right) +c_{1}\left( f_{\tilde{u}}\right) \varphi _{1}\left(
t\right) +c_{1}\left( f_{\tilde{u}}\right) \varphi _{1}\left( t\right)
+c_{2}\left( f_{\tilde{u}}\right) \varphi _{2}\left( t\right) +c_{3}\left(
f_{\tilde{u}}\right) \right) \right \vert \medskip  \notag \\
&&+\left \vert \theta _{\tilde{u}}\left( t,f\right) -\theta _{u^{\ast
	}}\left( t,f\right) \right \vert \medskip +\left \vert \left( c_{1}\left( f_{\tilde{u}}\right) -c_{1}^{\prime}\left( f_{u^{\ast }}\right) \right) \varphi _{1}\left( t\right) \right
	\vert \notag \\
 &&+\left \vert \left( c_{2}\left( f_{\tilde{u}}\right) -c_{2}^{\prime
	}\left( f_{u^{\ast }}\right) \right) \right \vert \varphi _{2}\left(
	t\right) +\left \vert c_{3}\left( f_{\tilde{u}}\right) -c_{3}^{\prime
	}\left( f_{u^{\ast }}\right) \right \vert .  \label{EQ136}
	\end{eqnarray}	
	By setting%
	\begin{equation}
	c_{{\small 33}}=c_{3}\left( f_{\tilde{u}}\right) -c_{3}^{\prime }\left(
	f_{u^{\ast }}\right) =\tilde{u}\left( a\right) -u^{\ast }\left( a\right) ,\
	c_{{\small 11}}=c_{1}\left( f_{\tilde{u}}\right) -c_{1}^{\prime }\left(
	f_{u^{\ast }}\right) \text{ and }c_{{\small 22}}=c_{2}\left( f_{\tilde{u}%
	}\right) -c_{2}^{\prime }\left( f_{u^{\ast }}\right) ,  \label{EQ137}
	\end{equation}%
	and%
	\begin{equation}
	w\left( \eta \right) =\tilde{u}\left( \eta \right) -u^{\ast }\left( \eta
	\right) +\theta _{\tilde{u}}\left( \eta ,f\right) -\theta _{u^{\ast }}\left(
	\eta ,f\right) \text{ and }w\left( T\right) =\tilde{u}\left( T\right)
	-u^{\ast }\left( T\right) +\theta _{\tilde{u}}\left( T,f\right) -\theta
	_{u^{\ast }}\left( T,f\right) .  \label{EQ138}
	\end{equation}	
	It follows from (\ref{EQ132}) and (\ref{EQ133}), that%
	\begin{equation}
	\left(
	\begin{array}{cc}
	\varphi _{1}\left( \eta \right) & \varphi _{2}\left( \eta \right) \\
	\varphi _{1}\left( T\right) & \varphi _{2}\left( T\right)%
	\end{array}%
	\right) \left(
	\begin{array}{c}
	c_{{\small 11}} \\
	c_{{\small 22}}%
	\end{array}%
	\right) =\left(
	\begin{array}{c}
	w\left( \eta \right) \\
	w\left( T\right)%
	\end{array}%
	\right) ,  \label{EQ139}
	\end{equation}
	Applying Lemma \ref{L9} and from estimation (\ref{EQ135}), it follows%
	\begin{equation}
	\left \vert \tilde{u}(t)-u^{\ast }(t)\right \vert \leq C_{\Phi }\left(
	t\right) +\left \vert \theta _{\tilde{u}}\left( t,f\right) -\theta _{u^{\ast
		}}\left( t,f\right) \right \vert +c_{11}\varphi _{1}\left( t\right)
		+c_{22}\varphi _{2}\left( t\right) +c_{33},  \label{EQ140}
		\end{equation}%
		where%
		\begin{eqnarray}
		\left \vert \theta _{\tilde{u}}\left( t,f\right) -\theta _{u^{\ast }}\left(
		t,f\right) \right \vert &=&\left \vert -\lambda \left( J_{a_{+},t}^{\varsigma
			,\psi }\right) \left[ \tilde{u}\right] +\left( J_{a_{+},t}^{\varrho +\varsigma
			,\psi }\right) \left[ f_{\tilde{u}}\right] -\left( -\lambda \left(
		J_{a_{+},t}^{\varsigma ,\psi }\right) \left[ u^{\ast }\right] +\left(
		J_{a_{+},t}^{\varrho +\varsigma ,\psi }\right) \left[ f_{u^{\ast }}\right]
		\right) \right \vert \medskip  \notag \\
		&\leq &\left \vert -\lambda \left( J_{a_{+},t}^{\varsigma ,\psi }\right) \left[
		\tilde{u}-u^{\ast }\right] +\left( J_{a_{+},t}^{\varrho +\varsigma ,\psi }\right) %
		\left[ f_{\tilde{u}}-f_{u^{\ast }}\right] \right \vert .  \label{EQ141}
		\end{eqnarray}
		Using Lemma \ref{L3} and (A$_{1}$), we have%
		\begin{equation}
		\left( J_{a_{+},t}^{\varrho +\varsigma ,\psi }\right) \left[ \left \vert \left(
		^{c}D_{a_{+},t}^{\delta ,\psi }\right) \left[ \tilde{u}-u^{\ast }\right]
		\right \vert \right] =z_{0}\left( t\right) -\left( J_{a_{+},t}^{\varrho
			+\varsigma -\delta ,\psi }\right) \left[ \tilde{u}-u^{\ast }\right]
		\label{EQ142}
		\end{equation}%
		and%
		\begin{equation}
		\left \vert \left( J_{a_{+},t}^{\varrho +\varsigma ,\psi }\right) \left[ f_{%
			\tilde{u}}-f_{u^{\ast }}\right] \right \vert \leq \left \vert L_{1}\left(
		J_{a_{+},t}^{\varrho +\varsigma ,\psi }\right) \left[ \tilde{u}-u^{\ast }\right]
		\right \vert +L_{2}\left \vert z_{0}\left( t\right) -\left(
		J_{a_{+},t}^{\varrho +\varsigma -\delta ,\psi }\right) \left[ \tilde{u}-u^{\ast }%
		\right] \right \vert ,  \label{EQ143}
		\end{equation}%
		where%
		\begin{equation}
		z_{0}\left( t\right) =\frac{\left \vert \tilde{u}\left( a\right) -u^{\ast
			}\left( a\right) \right \vert }{\Gamma \left( \varrho +\varsigma \right) }\times
		\left( J_{a,t-}^{1-\delta ,\psi }\right) \left[ \left( \psi \left( t\right)
		-\psi \left( a\right) \right) ^{\varrho +\varsigma -1}\right] .  \label{EQ144}
		\end{equation}
		Set%
		\begin{equation}
		q\left( t\right) =\mathcal{G}\left( t\right) +L_{2}\frac{\left \vert \tilde{u%
			}\left( a\right) -u^{\ast }\left( a\right) \right \vert }{\Gamma \left(
			\varrho +\varsigma \right) }\times \left( J_{a,t-}^{1-\delta ,\psi }\right) \left[
		\left( \psi \left( t\right) -\psi \left( a\right) \right) ^{\varrho +\varsigma -1}%
		\right] ,  \label{EQ145}
		\end{equation}%
		where%
		\begin{equation}
		\mathcal{G}\left( t\right) =C_{\Phi }\left( t\right) +c_{11}\varphi
		_{1}\left( t\right) +c_{22}\varphi _{2}\left( t\right) +c_{33}.
		\label{EQ146}
		\end{equation}%
		with%
		\begin{equation}
		C_{\Phi }\left( t\right) =\epsilon \left( J_{a_{+},t}^{\varrho +\varsigma ,\psi
		}\right) \left[ \Phi \right] +c_{1}\left( \epsilon \Phi \right) \varphi
		_{1}\left( t\right) +c_{2}\left( \epsilon \Phi \right) \varphi _{2}\left(
		t\right) +c_{3}\left( \epsilon \Phi \right) .  \label{EQ147}
		\end{equation}
		This means that%
		\begin{equation}
		p\left( t\right) \leq q\left( t\right) +\lambda \left( J_{a_{+},t}^{\varsigma
			,\psi }\right) \left[ \tilde{u}-u^{\ast }\right] +L_{1}\left(
		J_{a_{+},t}^{\varrho +\varsigma ,\psi }\right) \left[ \tilde{u}-u^{\ast }\right]
		+L_{2}\left( J_{a_{+},t}^{\varrho +\varsigma -\delta ,\psi }\right) \left[ \tilde{%
			u}-u^{\ast }\right] .  \label{EQ148}
		\end{equation}
		Using Lemma \ref{L6}, the above inequality implies the estimation for $%
		p\left( t\right) $ such as%
		\begin{equation}
		p\left( t\right) \leq q\left( t\right) +\sum_{k=1}^{\infty }\left(
		\begin{array}{l}
		\tfrac{\left( \lambda \Gamma \left( \varsigma \right) \right) ^{k}}{\Gamma
			\left( k\varsigma \right) }\int_{a}^{t}\left[ \psi ^{\prime }\left( \tau \right)
		\left( \psi \left( t\right) -\psi \left( \tau \right) \right) ^{k\varsigma -1}%
		\right] q\left( \tau \right) \mathrm{d}\tau \\
		+ \\
		\tfrac{\left( L_{1}\Gamma \left( \varrho +\varsigma \right) \right) ^{k}}{\Gamma
			\left( k\left( \varrho +\varsigma \right) \right) }\int_{a}^{t}\left[ \psi
		^{\prime }\left( \tau \right) \left( \psi \left( t\right) -\psi \left( \tau
		\right) \right) ^{k\left( \varrho +\varsigma \right) -1}\right] q\left( \tau
		\right) \mathrm{d}\tau \\
		+ \\
		\tfrac{\left( L_{2}\Gamma \left( \varrho +\varsigma -\delta \right) \right) ^{k}}{%
			\Gamma \left( k\left( \varrho +\varsigma -\delta \right) \right) }\int_{a}^{t}%
		\left[ \psi ^{\prime }\left( \tau \right) \left( \psi \left( t\right) -\psi
		\left( \tau \right) \right) ^{k\left( \varrho +\varsigma -\delta \right) -1}%
		\right] q\left( \tau \right) \mathrm{d}\tau%
		\end{array}%
		\right) .  \label{EQ149}
		\end{equation}
		Therefore, with (A$_{4}$), the inequality (\ref{EQ148}) can be rewritten as%
		\begin{equation}
		p\left( t\right) =\left \vert \tilde{u}(t)-u^{\ast }(t)\right \vert \leq
		\epsilon l_{\Phi }\Phi \left( t\right) .  \label{EQ150}
		\end{equation}
		By Remark \ref{R1}, one can obtain%
		\begin{eqnarray*}
		p\left( t\right) &\leq& q\left( t\right) \left[ E_{\varsigma ,\psi }\left( \lambda
		\Gamma \left( \varsigma \right) \left( \psi \left( t\right) \right) ^{\varsigma
		}\right) +E_{\varrho +\varsigma ,\psi }\left( \lambda \Gamma \left( \varrho +\varsigma
		\right) \left( \psi \left( t\right) \right) ^{\varrho +\varsigma }\right) \right.  \\
		&&+\left. E_{\varrho +\varsigma +\delta ,\psi }\left( \lambda \Gamma \left( \varrho +\varsigma
		+\delta \right) \left( \psi \left( t\right) \right) ^{\varrho +\varsigma +\delta
		}\right) \right] .
		\end{eqnarray*}
		Thus, we complete the proof.
	\end{proof}

\begin{theorem}
\label{T6}Assume that the assumptions (A$_{1}$) and (A$_{4}$). If a
continuously differentiable function $\tilde{u}:\left[ a,T\right]
\longrightarrow
\mathbb{R}
$ satisfies (\ref{EQ115}), where $\Phi :\left[ a,T\right] \longrightarrow
\mathbb{R}
^{+}$ is a continuous function with (A$_{3}$), then there exists a unique
continuous function $u^{\ast }:\left[ a,T\right] \longrightarrow
\mathbb{R}
$ of problem (\ref{EQ1}) such that
\begin{equation}
\left \vert \tilde{u}\left( t\right) -u^{\ast }\left( t\right) \right \vert
\leq \epsilon l_{\Phi }\Phi \left( t\right) ,  \label{EQ151}
\end{equation}
with
\begin{equation}
\left \vert \tilde{u}\left( a\right) -u^{\ast }\left( a\right) \right \vert
=\left \vert \tilde{u}\left( \eta \right) -u^{\ast }\left( \eta \right)
\right \vert =\left \vert \tilde{u}\left( T\right) -u^{\ast }\left( T\right)
\right \vert =0.  \label{EQ152}
\end{equation}
\end{theorem}

\begin{proof}
		Assume that $\tilde{u}\in \mathcal{C}^{3}\left( \left[ a,T\right] ,%
		\mathbb{R}
		\right) $ is a solution of the (\ref{EQ115}). In view of proof of Theorem %
		\ref{T5}, we get%
		\begin{equation}
		\mathcal{G}\left( t\right) =C_{\Phi }\left( t\right) +c_{11}\varphi
		_{1}\left( t\right) +c_{22}\varphi _{2}\left( t\right) +c_{33}=C_{\Phi
		}\left( t\right)  \label{EQ153}
		\end{equation}%
		with the conditions (\ref{EQ152}), we have
		\begin{equation}
		C_{\Phi }\left( t\right) =\epsilon \left \vert \left( J_{a_{+},t}^{\varrho
			+\varsigma ,\psi }\right) \left[ \Phi \right] +\left( J_{a_{+},\eta }^{\varrho
			+\varsigma ,\psi }\right) \left[ \Phi \right] d_{11}\left( t\right) +\left(
		\left( J_{a_{+},T}^{\varrho +\varsigma ,\psi }\right) \left[ \Phi \right] -\mu
		\left( J_{a_{+},\xi }^{\varrho +\varsigma +\delta ,\psi }\right) \left[ \Phi %
		\right] \right) d_{12}\left( t\right) \right \vert .  \label{EQ154}
		\end{equation}
		Set $q\left( t\right) =C_{\Phi }\left( t\right) .\ $Using Theorem \ref{T5}
		and (A$_{4}$), we conclude that, the estimation for $p\left( t\right) =\left
		\vert u\left( t\right) -\tilde{u}\left( t\right) \right \vert $ such as (\ref%
		{EQ148}).\ So the inequality (\ref{EQ148}) can be rewritten as%
		\begin{equation}
		p\left( t\right) =\left \vert u\left( t\right) -\tilde{u}\left( t\right)
		\right \vert \leq \epsilon l_{\Phi }\Phi \left( t\right) .  \label{EQ156}
		\end{equation}
		By Remark \ref{R1}, one can obtain%
		\begin{eqnarray*}
		p\left( t\right) &\leq& q\left( t\right) \left[ E_{\varsigma ,\psi }\left( \lambda
		\Gamma \left( \varsigma \right) \left( \psi \left( t\right) \right) ^{\varsigma
		}\right) +E_{\varrho +\varsigma ,\psi }\left( \lambda \Gamma \left( \varrho +\varsigma
		\right) \left( \psi \left( t\right) \right) ^{\varrho +\varsigma }\right) \right. \\
		&&+ \left. E_{\varrho +\varsigma +\delta ,\psi }\left( \lambda \Gamma \left( \varrho +\varsigma
		+\delta \right) \left( \psi \left( t\right) \right) ^{\varrho +\varsigma +\delta
		}\right) \right].
		\end{eqnarray*}
		This proves that the problem (\ref{EQ1}) is H-U-R stable.
	\end{proof}

\begin{theorem}
\label{T7}Assume that the assumptions (A$_{2}$), (A$_{4}$) and (\ref{EQ115})
hold. Then the equation (\ref{EQ1}-a) is H-U-R stable.
\end{theorem}

\begin{proof}
		By (A$_{2}$) and (\ref{EQ135}), we have%
		\begin{equation}
		\left \vert \left( J_{a_{+},t}^{\varrho +\varsigma ,\psi }\right) \left[ f_{%
			\tilde{u}}\right] -\left( J_{a_{+},t}^{\varrho +\varsigma ,\psi }\right) \left[
		f_{u}\right] \right \vert \leq \left \vert \left( J_{a_{+},t}^{\varrho +\varsigma
			,\psi }\right) \left[ \tilde{\chi}\right] -\left( J_{a_{+},t}^{\varrho +\varsigma
			,\psi }\right) \left[ \chi \right] \right \vert  \label{EQ157}
		\end{equation}%
		and%
		\begin{equation}
		\left \vert \tilde{u}(t)-u^{\ast }(t)\right \vert \leq C_{\Phi }\left(
		t\right) +\left \vert \theta _{\tilde{u}}\left( t,f+g\right) -\theta
		_{u^{\ast }}\left( t,f\right) \right \vert +c_{11}\left \vert \varphi
		_{1}\left( t\right) \right \vert +c_{22}\left \vert \varphi _{2}\left(
		t\right) \right \vert +c_{33},  \label{EQ158}
		\end{equation}%
		where%
		\begin{eqnarray}
		&&\left \vert \theta _{\tilde{u}}\left( t,f+g\right) -\theta _{u^{\ast
			}}\left( t,f\right) \right \vert \notag \\
        &&=\left \vert -\lambda \left(
			J_{a_{+},t}^{\varsigma ,\psi }\right) \left[ \tilde{u}\right] +\left(
			J_{a_{+},t}^{\varrho +\varsigma ,\psi }\right) \left[ f_{\tilde{u}}+g\right]
			-\left( -\lambda \left( J_{a_{+},t}^{\varsigma ,\psi }\right) \left[ u^{\ast }%
			\right] +\left( J_{a_{+},t}^{\varrho +\varsigma ,\psi }\right) \left[ f_{u^{\ast
				}}\right] \right) \right \vert \medskip  \notag \\
				&&\leq \left \vert -\lambda \left( J_{a_{+},t}^{\varsigma ,\psi }\right) \left[
				\tilde{u}-u^{\ast }\right] +\left( J_{a_{+},t}^{\varrho +\varsigma ,\psi }\right) %
				\left[ f_{\tilde{u}}-f_{u^{\ast }}\right] \right \vert +\left \vert \left(
				J_{a_{+},t}^{\varrho +\varsigma ,\psi }\right) \left[ g\right] \right \vert .
				\label{EQ159}
				\end{eqnarray}
				Using Lemma \ref{L9}, we have%
				\begin{equation}
				p\left( t\right) \leq q\left( t\right) +\lambda \left( J_{a_{+},t}^{\varsigma
					,\psi }\right) \left[ \tilde{u}-u^{\ast }\right] ,  \label{EQ160}
				\end{equation}%
				where%
				\begin{equation}
				q\left( t\right) =\mathcal{G}\left( t\right) +\left \vert \left(
				J_{a_{+},t}^{\varrho +\varsigma ,\psi }\right) \left[ \tilde{\chi}\right] -\left(
				J_{a_{+},t}^{\varrho +\varsigma ,\psi }\right) \left[ \chi \right] \right \vert ,
				\label{EQ161}
				\end{equation}%
				with%
				\begin{equation}
				\mathcal{G}\left( t\right) =C_{\Phi }\left( t\right) +c_{11}\varphi
				_{1}\left( t\right) +c_{22}\varphi _{2}\left( t\right) +c_{33}.
				\label{EQ162}
				\end{equation}
				From the above, it follows%
				\begin{equation}
				p\left( t\right) \leq q\left( t\right) +\sum_{k=1}^{\infty }%
				\begin{array}{l}
				\dfrac{\left( \lambda \Gamma \left( \varsigma \right) \right) ^{k}}{\Gamma
					\left( k\varsigma \right) }\int_{a}^{t}\left[ \psi ^{\prime }\left( \tau \right)
				\left( \psi \left( t\right) -\psi \left( \tau \right) \right) ^{k\varsigma -1}%
				\right] q\left( \tau \right) \mathrm{d}\tau%
				\end{array}%
				.  \label{EQ163}
				\end{equation}
				By Remark \ref{R1}, one can obtain%
				\begin{equation*}
				p\left( t\right) \leq q\left( t\right) E_{\varsigma ,\psi }\left( \lambda \Gamma
				\left( \varsigma \right) \left( \psi \left( t\right) \right) ^{\varsigma }\right) .
				\end{equation*}
			\end{proof}

\begin{remark}
\label{R3}If $\Phi \left( t\right) $ is a constant function in the
inequalities (\ref{EQ115}), the we say that (\ref{EQ1}-a) is H-U
stable.
\end{remark}

\begin{corollary}
\label{C2}Assume that the assumptions (A$_{2}$), (A$_{4}$) and (\ref{EQ115})
hold. Then the equation (\ref{EQ1}-a) with (\ref{EQ151}) is
H-U-R stable.
\end{corollary}

\begin{proof}
				Using Theorem \ref{T7}, we have%
				\begin{equation}
				p\left( t\right) \leq q\left( t\right) +\lambda \left( J_{a_{+},t}^{\varsigma
					,\psi }\right) \left[ \tilde{u}-u^{\ast }\right] ,  \label{EQ164}
				\end{equation}%
				where%
				\begin{equation}
				p\left( t\right) =\left \vert \tilde{u}\left( t\right) -u\left( t\right)
				\right \vert \text{ and}\ q\left( t\right) =C_{\Phi }\left( t\right) +\left
				\vert \left( J_{a_{+},t}^{\varrho +\varsigma ,\psi }\right) \left[ \tilde{\chi}%
				\right] -\left( J_{a_{+},t}^{\varrho +\varsigma ,\psi }\right) \left[ \chi \right]
				\right \vert .  \label{EQ165}
				\end{equation}
				We conclude that%
				\begin{equation}
				p\left( t\right) \leq q\left( t\right) +\sum_{k=1}^{\infty }%
				\begin{array}{l}
				\dfrac{\left( \lambda \Gamma \left( \varsigma \right) \right) ^{k}}{\Gamma
					\left( k\varsigma \right) }\int_{a}^{t}\left[ \psi ^{\prime }\left( \tau \right)
				\left( \psi \left( t\right) -\psi \left( \tau \right) \right) ^{k\varsigma -1}%
				\right] q\left( \tau \right) \mathrm{d}\tau%
				\end{array}%
				.  \label{EQ166}
				\end{equation}
				By Remark \ref{R1}, one can obtain%
				\begin{equation*}
				p\left( t\right) \leq q\left( t\right) E_{\varsigma ,\psi }\left( \lambda \Gamma
				\left( \varsigma \right) \left( \psi \left( t\right) \right) ^{\varsigma }\right) .
				\end{equation*}
			\end{proof}

\begin{theorem}
\label{T8}Assume that the assumptions (A$_{1}$) and (\ref{EQ115}) with (\ref%
{EQ152}) hold. Then the problem (\ref{EQ1}) is U-H stable and
consequently generalized U-H stable.
\end{theorem}

\begin{proof}
				Let $u^{\ast }$ be a unique solution of the fractional Langevin type problem
				(\ref{EQ1}), that is, $u^{\ast }(t)=\left( \Psi u^{\ast }\right) (t)$.
				Assume that $\tilde{u}\in \mathcal{C}\left( \left[ a,T\right] ,%
				\mathbb{R}
				\right) $ is a solution of the (\ref{EQ115})
				By using the estimation
				\begin{equation}
				\left \vert \left( \Psi \tilde{u}\right) \left( t\right) -\left( \Psi
				u^{\ast }\right) \left( t\right) \right \vert \leq \lambda \left \vert
				\left( J_{a_{+},t}^{\varsigma ,\psi }\right) \left[ \tilde{u}-u^{\ast }\right]
				\right \vert +\left \vert \left( J_{a_{+},t}^{\varrho +\varsigma ,\psi }\right) %
				\left[ f_{u^{\ast }}-\left( f_{\tilde{u}}+g\right) \right] \right \vert
				+\left \vert \phi _{\tilde{u}}\left( f+g\right) -\phi _{u^{\ast }}\left(
				t,f\right) \right \vert ,  \label{EQ168}
				\end{equation}%
				where%
				\begin{eqnarray}
				&&\left \vert \phi _{\tilde{u}}\left( f+g\right) -\phi _{u^{\ast }}\left(
				t,f\right) \right \vert \notag \\
&&=d_{11}\left( t\right) \left( J_{a_{+},\eta
				}^{\varrho +\varsigma ,\psi }\right) \left[ f_{\tilde{u}}-f_{u^{\ast }}\right]
				+d_{12}\left( t\right) \left( \left( J_{a_{+},T}^{\varrho +\varsigma ,\psi
				}\right) \left[ f_{\tilde{u}}-f_{u^{\ast }}\right] -\mu \left( J_{a_{+},\xi
			}^{\varrho +\varsigma +\delta ,\psi }\right) \left[ f_{\tilde{u}}-f_{u^{\ast }}%
			\right] \right) \medskip  \notag \\
			&&+\lambda d_{21}\left( t\right) \left( J_{a_{+},\eta }^{\varsigma ,\psi
			}\right) \left[ \tilde{u}-u^{\ast }\right] -\lambda d_{22}\left( t\right)
			\left( \left( J_{a_{+},T}^{\varsigma ,\psi }\right) \left[ \tilde{u}-u^{\ast }%
			\right] -\mu \left( J_{a_{+},\xi }^{\varsigma +\delta ,\psi }\right) \left[
			\tilde{u}-u^{\ast }\right] \right) \medskip  \notag \\
			&&+d_{11}\left( t\right) \left( J_{a_{+},\eta }^{\varrho +\varsigma ,\psi
			}\right) \left[ g\right] +d_{12}\left( t\right) \left( \left(
			J_{a_{+},T}^{\varrho +\varsigma ,\psi }\right) \left[ g\right] -\mu \left(
			J_{a_{+},\xi }^{\varrho +\varsigma +\delta ,\psi }\right) \left[ g\right] \right).
			 \label{EQ169}
			\end{eqnarray}
			Taking the maximum over $\left[ a,T\right] ,$ we get%
			\begin{equation}
			\sup_{t\in \left[ a,T\right] }\left \vert \left( \Psi \tilde{u}\right)
			\left( t\right) -\left( \Psi u^{\ast }\right) \left( t\right) \right \vert
			\leq \varsigma _{11}\sup_{t\in \left[ a,T\right] }\left \vert \tilde{u}%
			\left( t\right) -u^{\ast }\left( t\right) \right \vert +\varsigma
			_{12}\sup_{t\in \left[ a,T\right] }\left \vert \left( ^{c}D_{a+,t}^{\delta
				,\psi }\right) \left[ \tilde{u}-u^{\ast }\right] \right \vert +\epsilon
			\varsigma _{13},  \label{EQ170}
			\end{equation}
			Using Lemma \ref{L2}\ and (\ref{EQ169}), we obtain%
			\begin{equation}
			\sup_{t\in \left[ a,T\right] }\left \vert \left( \Psi \tilde{u}\right)
			\left( t\right) -\left( \Psi u^{\ast }\right) \left( t\right) \right \vert
			\leq \left( \varsigma _{11}+\varsigma _{12}\kappa _{0}\right) \sup_{t\in %
				\left[ a,T\right] }\left \vert \tilde{u}\left( t\right) -u^{\ast }\left(
			t\right) \right \vert +\epsilon \varsigma _{13},  \label{EQ172}
			\end{equation}%
			where%
			\begin{equation}
			\kappa _{0}=\frac{1}{\Gamma \left( 2-\delta \right) }\left( \psi \left(
			T\right) -\psi \left( a\right) \right) ^{1-\delta };  \label{EQ173}
			\end{equation}
			We conclude that%
			\begin{equation}
			\left \Vert \tilde{u}-u^{\ast }\right \Vert _{\infty }\leq \frac{\epsilon
				\varsigma _{13}}{\left( 1-\varsigma _{11}-\varsigma _{12}\kappa _{0}\right) }%
			,\ 0<1-\varsigma _{11}-\varsigma _{12}\kappa _{0}<1.  \label{EQ174}
			\end{equation}
			Thus problem (\ref{EQ1}) is U-H stable. Further, using implies that
			solution of (\ref{EQ1}) is generalized U-H stable. This completes
			the proof.
		\end{proof}

\begin{corollary}
\label{C3}Let the conditions of Theorem \ref{T8} hold. Then the problem (\ref%
{EQ1}) is generalized U-H-R stable.
\end{corollary}

\begin{proof}
			Set $\epsilon =1$ in the proof of Theorem \ref{T8}, we get%
			\begin{equation}
			\left \Vert \tilde{u}-u^{\ast }\right \Vert _{\infty }\leq \frac{\varsigma
				_{13}}{\left( 1-\varsigma _{11}-\varsigma _{12}\kappa _{0}\right) },\  \
			0<1-\varsigma _{11}-\varsigma _{12}\kappa _{0}<1.  \label{EQ175}
			\end{equation}
		\end{proof}

\begin{remark}
\label{RR4}
\begin{itemize}
  \item [(\textbf{i})] Under the assumptions of Theorem \ref{T6}, we
consider (\ref{EQ1}) and the inequality (\ref{EQ115}), one can renew the same procedure to confirm that (\ref{EQ1}) is U-H stable.
  \item [(\textbf{ii})] Other stability results for the equation (\ref{EQ1}) can be
discussed similarly.
\end{itemize}
\end{remark}
\section{Applications}

Will be provided in the revised submission.

\section{Conclusion}

\quad The Langevin equation has been introduced to characterize dynamical processes in a fractal medium in which the fractal and memory features with a dissipative memory kernel are incorporated. Therefore, the consideration of Langevin equation in frame of fractional derivatives settings would be providing better interpretation for real phenomena. Consequently, scholars have considered different versions of Langevin equation and thus many interesting papers have been reported in this regard. However, one can notice that most of existing results have been carried out with respect to the classical fractional derivatives.

In this paper, we have tried to promote the current results and considered the FLE in a general platform. The boundary value problem of nonlinear FLE involving $\psi $- fractional operator was investigated.
We employ the the newly accommodated $\psi $- fractional calculus to prove the following for the considered problem:
\begin{itemize}
  \item [(i.)] The existence and
uniqueness of solutions: Techniques of fixed point theorems are used to prove the results. Prior to the main theorems, the forms of solutions are derived for both linear and nonlinear problems.
  \item [(ii.)] Stability in sense of Ulam: We adopt the required definitions of U-H stability with respect to $\psi $- fractional derivative. The U-H-R and generalized U-H-R stability of the
solution are discussed. Gronwall inequality and integration by parts in frame of $\psi $- fractional drivative are also employed to complete the proofs.
 \item [(iii.)] Applications: Couple of particular examples are addressed at the end of the paper to show the consistency of the theoretical results.
\end{itemize}
We claim that the results of this paper are new and generalize some earlier results.
For further investigation, one can propose to study the properties of the solution of the considered problem via some numerical computations and simulations. We leave this as promising future
work. Results obtained in the present paper can be considered as a
contribution to the developing field of fractional calculus via generalized
fractional derivative operators.
\section*{Availability of data and material}
Not applicable.
\section*{Competing Interests}
The authors declare that they have no competing interests.
\section*{Funding} Not applicable.
\section*{Author's contributions}
All authors contributed equally and significantly to this paper. All authors have read and approved the final version of the manuscript.
\section*{Acknowledgement}
J. Alzabut would like to thank Prince Sultan University for
funding this work through research group Nonlinear Analysis Methods
in Applied Mathematics (NAMAM) group number RG-DES-2017-01-17.

\end{document}